\documentclass{amsart}
\usepackage{amsfonts,amssymb,amscd,amsmath,enumerate,verbatim,calc}
\usepackage{tikz-cd}
\usepackage[all]{xy}
 \usepackage{hyperref}
\hypersetup{colorlinks,linkcolor={blue},citecolor={red},urlcolor={red}}  
\tikzset{mynode/.style={inner sep=2pt,fill,outer sep=0,circle}}

\newcommand{\m}{\mathfrak{m} }

\providecommand{\D}{{\mathcal D}}

\newcommand{\M}{\mathcal{M}}

\newcommand{\Ass}{\operatorname{Ass}}

\newcommand{\gdepth}{\operatorname{g-depth}}
\newcommand{\depth}{\operatorname{depth}}
\newcommand{\Spec}{\operatorname{Spec}}

\newcommand{\ann}{\operatorname{ann}}

\newcommand{\bideg}{\operatorname{bideg}}
\newcommand{\height}{\operatorname{height}}

\newcommand{\Bmod}{\operatorname{Bmod}}
\newcommand{\Bproj}{\operatorname{Bproj}}

\newcommand{\Mod}{\operatorname{Mod}}

\newcommand{\injdim}{\operatorname{injdim}}

\newcommand{\Ext}{\operatorname{Ext}}

\newcommand{\Bfg}{\operatorname{B-fg}}
\newcommand{\Bfgt}{\operatorname{B*-fg}}
\newcommand{\Supp}{\operatorname{Supp}}

\theoremstyle{plain}

\newtheorem{theorem}{Theorem}[section]
\newtheorem{corollary}[theorem]{Corollary}
\newtheorem{lemma}[theorem]{Lemma}
\newtheorem{proposition}[theorem]{Proposition}

\newtheorem{maintheorema}{ Theorem A}

\newtheorem{maintheoremb}{ Theorem B}

\newtheorem{maintheoremc}{ Theorem C}

\newtheorem{maintheoremd}{ Theorem D}

\newtheorem{maintheoreme}{ Theorem E}

\newtheorem{maintheoremf}{ Theorem F}

\newtheorem{maintheoremg}{ Theorem G}

\newtheoremstyle{pictoStyle}
  {0pt}
  {0pt} 
  {\normalfont} 
  {0pt} 
  {\bfseries}
  {} 
  {1em}
  {} 

\theoremstyle{pictoStyle}
\newtheorem*{picto}{Pictorial Description :} 

\theoremstyle{definition}
\newtheorem{definition}[theorem]{Definition}
\newtheorem{s}[theorem]{}
\newtheorem{remark}[theorem]{Remark}
\newtheorem{example}[theorem]{Example}

\theoremstyle{remark}

\begin{document}

\title{Bigraded components of F-finite F-modules }
\address{department of mathematics, indian institute of technology bombay, powai, mumbai 400076, india}
 \author{Sayed Sadiqul Islam }
 \email{22d0786@iitb.ac.in, ssislam1997@gmail.com}
 \address{department of mathematics, indian institute of technology bombay, powai, mumbai 400076, india}
 \author{Tony J. Puthenpurakal}
 \email{tputhen@math.iitb.ac.in}

\date{\today}

\date{\today}

\begin{abstract}
    Let $A$ be a regular ring containing a field of characteristic $p>0$ and let $R=A[x_1,\ldots,x_m,y_1,\ldots,y_n]$ be standard bigraded over $A$, i.e., $\bideg(A)=(0,0)$, $\bideg(x_i)=(1,0)$ and  $\bideg(y_j)=(0,1)$ for all $i$ and $j$. Assume that $M=\bigoplus_{i,j} M_{(i,j)}$ is a bigraded $F_R$-finite, $F_R$-module. We use Lyubeznik's theory of $F$-finite, $F$-modules from \cite{Lyu-Fmod} to study the bigraded components of $M$. The properties we study include vanishing, rigidity, Bass numbers, associated primes, and injective dimension of the components of $M$. As an application we show that if $(A,\mathfrak{m})$ is regular local ring containing a field of characteristic $p>0$, $R/I$ is equidimensional,  $\Bproj(R/I)$ is Cohen-Macaulay and non-empty, then $H^j_I(R)_{(m,n)}=0$ for all $(m,n)\geq (0,0)$ and all $j>\height I$.
\end{abstract}

\maketitle
{\vspace{-3em}
\setlength{\parskip}{0em}
\tableofcontents
\vspace{-3em}}

\section{Introduction}
Local cohomology was discovered in 1960s as a tool to study sheaves and their cohomology in algebraic geometry, but has since found wide application in commutative algebra. Over time, it has become an indispensable tool and a major area of research. Local cohomology modules are not finitely generated in general. For example, if $(R,\mathfrak{m})$ is a local ring of dimension $d\geq 1$, then $H^d_\mathfrak{m}(R)$ is never finitely generated. 

Let $S=\oplus_{n\geq 0}S_n$ be standard Noetherian ring and $S_+$ be irrelevant ideal. The theory of local cohomology  with respect to $S_+$ is particularly well behaved \cite[Theorem 15.1.5]{BS}. It is well-known that if $M$ is finitely generated graded $S$-module then for $i\geq 0$,
\begin{enumerate}
    \item $H^i_{S_+}(M)_n$ is finitely generated $S_0$-module for all $n\in \mathbb{Z}$.
    \item $H^i_{S_+}(M)_n=0$ for all $n\gg 0$.
\end{enumerate}
Motivated by this fact, the second author studied the graded components (with respect to arbitrary homogenenous ideal of $R$) of local cohomology module of $R=A[x_1,\ldots,x_n]$  where $A$ is a regular ring containing a field of characteristic 0 (see \cite{TP-collect}). As a tool, the author uses the notion of generalized Eulerian $\D$-modules (in characteristic zero), see \cite{TP-Nagoya}. He introduced generalized Eulerian $\D$-modules as an extension of Eulerian modules, which was introduced in the case where $A$ is a field $K$ by Ma and Zhang \cite{Ma}. In \cite{TP-viet}, he, together with  co-author, derived some of these results under less stringent assumptions on the base ring $A$. Later (see \cite{TP-comm}) they studied the case when $A=B^G$, the ring of invariants of a regular domain $B$ that contains a field $K$ of characteristic zero and $G$ is a finite subgroup of the automorphism group of $B$. It is natural to inquire whether the results on graded local cohomology obtained in characteristic zero can be extended to the case when $p>0$. To address this, in \cite{TP-koszul}, the second author utilizes the theory of $F$-modules and proved many similar results. The concept of $F$-modules was introduced by Lyubeznik in \cite{Lyu-Fmod}, where he showed that $F$-finite, $F$-modules satisfy strong finiteness conditions. This conditions include finiteness of Bass numbers and the set of associated primes.  

Let $A$ be a regular ring containing a field of characteristic $p>0$. Let $R=A[x_1,\ldots,x_m,y_1,\ldots,y_n]$ be standard bigraded over $A$. In this paper we use the theory of $F$-modules to study the bigraded components of local cohomology modules of $R$. In this paper we generalize most of the results proved in \cite{TP-koszul} to the bigraded setup. Many of the results of our paper holds true when $A$ is a regular ring containing a field of characteristic $0$. These results are proved in \cite{TP-big} which is currently under preparation. The authors in \cite{TP-big} use the theory of $\D$-modules and generalized Eulerian $\D$-modules to prove their result.

We now state the results proved in this paper.
\s\textit{Setup:}\label{Main setup}
Let $A$ be a regular ring containing a field of characteristic $p>0$ and $R=A[x_1,\ldots,x_m,y_1,\ldots,y_n]$ be standard bigraded over $A$, i.e., $\bideg(A)=(0,0)$, $\bideg(x_i)=(1,0)$ and  $\bideg(y_j)=(0,1)$ for all $i$ and $j$. Let $M=\bigoplus_{i,j} M_{(i,j)}$ be a bigraded $F_R$-finite, $F_R$-module.

The first result we prove is the vanishing of almost bigraded components of $M$ implies vanishing of $M$. More precisely we prove the following:
\begin{maintheorema}[Theorem \ref{Vanishing of M-A}]
\label{thm:maintheorema}
  Let $M$ be as in \ref{Main setup}. If $M_{(m,n)}=0$ for all $|m|\gg 0 $ and $|n|\gg 0$, then $M=0$.
\end{maintheorema}
Let us denote $\underline{u}=(u_1, \dots, u_r)\in \mathbb{Z}^r$ and $\underline{v}=(v_1,\ldots,v_r)\in \mathbb{Z}^r$. We say that $\underline{v}\geq \underline{u}$ if $v_i\geq u_i$ for all $i= 1, \ldots, r$. In the next two results we establish the rigidity properties of the bigraded components of $M$.
\begin{maintheoremb}[Theorem \ref{Rigidity for A-1}] 
\label{thm:maintheoremb}With hypothesis as in \ref{Main setup}, we have
 \begin{enumerate}[\rm (1)]
     \item The following statements are equivalent:
     \begin{enumerate}[\rm (a)]
     \item $M_{(a,b)}\neq 0$ for all $(a,b)\geq (0,0)$.
         \item There exists $(a_0,b_0)\geq (0,0)$ such that $M_{(a,b)}\neq 0$ for all $(a,b)\geq (a_0,b_0)$.
         \item $M_{(a_1,b_1)}\neq 0$ for some $(a_1,b_1)\geq (0,0)$.
     \end{enumerate}
     \item The following statements are equivalent:
     \begin{enumerate}[\rm (a)]
         \item $M_{(a,b)}\neq 0$ for all $(a,b)\in\mathbb{Z}^2$ with $a\leq -m$ and $b\geq 0$.
         \item There exists $(a_0,b_0)\in \mathbb{Z}^2$ with $a_0\leq -m$ and $b_0\geq 0$ such that $M_{(a,b)}\neq 0$ for all $(a,b)\in \mathbb{Z}^2$ with $a\leq a_0$ and $b\geq b_0$.
         \item $M_{(a_1,b_1)}\neq 0$ for some $(a_1,b_1)\in \mathbb{Z}^2$ with $a_1\leq -m$ and $b_1\geq 0$.
     \end{enumerate}
     \item The following statements are equivalent:
     \begin{enumerate}[\rm (a)]
         \item $M_{(a,b)}\neq 0$ for all $(a,b)\in\mathbb{Z}^2$ with $(a,b)\leq (-m,-n)$.
         \item There exists $(a_0,b_0)\in \mathbb{Z}^2$ with $(a_0,b_0)\leq (-m,-n)$  such that $M_{(a,b)}\neq 0$ for all $(a,b)\leq (a_0,b_0)$.
         \item $M_{(a_1,b_1)}\neq 0$ for some $(a_1,b_1)\in \mathbb{Z}^2$ with $(a_1,b_1)\leq (-m,-n)$.
     \end{enumerate}
\item The following statements are equivalent:
     \begin{enumerate}[\rm (a)]
         \item $M_{(a,b)}\neq 0$ for all $(a,b)\in\mathbb{Z}^2$ with $a\geq 0$ and $b\leq -n$.
         \item There exists $(a_0,b_0)\in \mathbb{Z}^2$ with $a_0\geq 0$ and $b_0\leq -n$ such that $M_{(a,b)}\neq 0$ for all $(a,b)\in \mathbb{Z}^2$ with $a\geq a_0$ and $b\leq b_0$.
         \item $M_{(a_1,b_1)}\neq 0$ for some $(a_1,b_1)\in \mathbb{Z}^2$ with $a_1\geq 0$ and $b_1\leq -n$.
     \end{enumerate}
 \end{enumerate}
\end{maintheoremb}
The following theorem examines the case when a component, excluded in theorem \hyperref[thm:maintheoremb]{Theorem B}, is non-zero.
\begin{maintheoremc}[Theorem \ref{Rigidity for A-2}]\label{thm:maintheoremc}
    With hypothesis as in \ref{Main setup}, we have
		\begin{enumerate}[\rm(1)]
		\item The following statements are equivalent:
		\begin{enumerate}[\rm(a)]
		\item $M_{(a,b)} \neq 0$ for all $(a,b)\in \mathbb{Z}^2$.
		\item There exists some $(-m,-n)<(a_0,b_0)<(0,0)$ such that $M_{(a_0,b_0)} \neq 0$.
	\end{enumerate}
\item The following statements are equivalent:
\begin{enumerate}[\rm (a)]
	\item $M_{(a,b)} \neq 0$ for all $(a,b)\in \mathbb{Z}^2$ with $b \geq 0$.
	\item There exists some $(a_0,b_0)\in \mathbb{Z}^2$ with $-m<a_0<0$ and $b_0 \geq 0$ such that $M_{(a_0,b_0)} \neq 0$.
\end{enumerate}
\item The following statements are equivalent:
\begin{enumerate}[\rm (a)]
	\item $M_{(a,b)} \neq 0$ for all $(a,b)\in \mathbb{Z}^2$ with $a \leq -m$.
	\item There exists some $(a_0,b_0)\in \mathbb{Z}^2$ with $a_0 \leq -m$ and $-n<b_0<0$ such that $M_{(a_0,b_0)} \neq 0$.
\end{enumerate}
\item The following statements are equivalent:
\begin{enumerate}[\rm (a)]
	\item $M_{(a,b)} \neq 0$ for all $(a,b)\in \mathbb{Z}^2$ with $b \leq -n$.
	\item There exists some $(a_0,b_0)\in \mathbb{Z}^2$ with $-m< a_0 <0$ and $b_0 \leq -n$ such that $M_{(a_0,b_0)} \neq 0$.
\end{enumerate}
\item The following statements are equivalent:
\begin{enumerate}[\rm (a)]
	\item $M_{(a,b)} \neq 0$ for all $(a,b)\in \mathbb{Z}^2$ with $a \geq 0$.
	\item There exists some $(a_0,b_0)\in \mathbb{Z}^2$ with $a_0 \geq 0$ and $-n<b_0<0$ such that $M_{(a_0,b_0)} \neq 0$.
\end{enumerate}
\end{enumerate}
\end{maintheoremc}
 By $\mu_i(P,M_{(a,b)})$, we denote the $i$-th Bass number of $M_{(a,b)}$ for some prime ideal $P$ of $A$. We also investigate  how the finiteness of $\mu_i(P,M_{(a,b)})$ for some pair $(a,b)$ relates to the finiteness of Bass numbers of other components. This is shown next.
 \begin{maintheoremd}[Theorem \ref{Bass th for 1st quad}]\label{thm:maintheoremd}
     Let the hypothesis be as in \ref{Main setup}. Then
   the following statements are true;
    \begin{enumerate}[\rm (1)]
        \item  If  $\mu_j(P,M_{(a,b)})$ is finite for some $(a,b)\in \mathbb{Z}^2$ such that $(a,b)\geq(0,0)$, then $ \mu_j(P,M_{(r,s)})$ is finite for all $(r,s)\geq (0,0)$. 
        \item  If $\mu_j(P,M_{(a,b)})$ is finite for some $(a,b)\in \mathbb{Z}^2$ such that $a<0$ and $b\geq 0$, then $ \mu_j(P,M_{(r,s)})$ is finite for all $(r,s)\in\mathbb{Z}^2$ with $r<0$ and $s\geq 0$.
        \item If  $\mu_j(P,M_{(a,b)})$ is finite for some $(a,b)\in \mathbb{Z}^2$ such that $(a,b)<(0,0)$, then $ \mu_j(P,M_{(r,s)})$ is finite for all $(r,s)<(0,0)$.
        \item  If  $\mu_j(P,M_{(a,b)})$ is finite for some  $(a,b)\in \mathbb{Z}^2$ such that $a\geq 0$ and $b<0$, then $ \mu_j(P,M_{(r,s)})$ is finite for all $(r,s)\in \mathbb{Z}^2$ with $r\geq 0$ and $s<0$.
        \end{enumerate}
\end{maintheoremd}
We also prove several finiteness and stability properties concerning the asymptotic primes of $M$.
 \begin{maintheoreme}[Theorem \ref{Associated prime result}]\label{thm:maintheoreme}
	 With hypothesis as in \ref{Main setup}, the following statements are true; 
	
	\begin{enumerate}[\rm (1)]
		\item
		$\bigcup_{(a,b) \in \mathbb{Z}^2} \Ass_A M_{(a,b)}   $ is a finite set.
		\item
		$\Ass_A M_{(a,b)} = \Ass_A M_{(0,0)}$ for all $(a,b) \geq (0,0)$.
		\item
		$\Ass_A M_{(a,b)} = \Ass_A M_{(-m,0)}$ for all $(a,b) \in \mathbb{Z}^2$ with $a \leq -m$ and $b \geq 0$.
		\item
		$\Ass_A M_{(a,b)} = \Ass_A M_{(-m,-n)}$ for all $(a,b)\leq (-m,-n)$.
		\item
		$\Ass_A M_{(a,b)} = \Ass_A M_{(0,-n)}$ for all $(a,b) \in \mathbb{Z}^2$ with $a \geq 0$ and $b \leq -n$.
	\end{enumerate}
 \end{maintheoreme}
Let $E$ be an $A$-module. By $\injdim_A E$, we denote the injective dimension of $E$ as $A$-module. By $\dim_AE$, we mean the dimension of the support as a subspace of $\Spec(A)$. Our result regarding injective dimension and dimension of bigraded components is the following:
\begin{maintheoremf}[Theroem \ref{injdim-and-dim-gen}]\label{thm:maintheoremf}
With hypotheses as in \ref{Main setup}, the following statements are true; 
	\begin{enumerate}[\rm (1)]
		\item
		$\injdim M_{(a,b)} \leq \dim M_{(a,b)}$ for all $(a,b) \in \mathbb{Z}^2$.
		\item
		$\injdim M_{(a,b)} = \injdim M_{(0,0)}$ for all $(a,b)\geq (0,0)$.
		\item $\injdim M_{(a,b)} = \injdim M_{(-m,0)}$ for all $(a,b) \in \mathbb{Z}^2$ with $a \leq -m$ and $b \geq 0$.
		\item $\injdim M_{(a,b)} = \injdim M_{(-m,-n)}$ for all $(a,b)\leq (-m,-n)$.
		\item $\injdim M_{(a,b)} = \injdim M_{(0,-n)}$ for all $(a,b) \in \mathbb{Z}^2$ with $a \geq 0$ and $b \leq -n$.
		\item
		$\dim M_{(a,b)} = \dim M_{(0,0)}$ for all $(a,b)\geq (0,0)$.
		\item
		$\dim M_{(a,b)} = M_{(-m,0)}$ for all $(a,b) \in \mathbb{Z}^2$ with $a \leq -m$ and $b \geq 0$.
		\item
	    $\dim M_{(a,b)} = M_{(-m,-n)}$ for all $(a,b)\leq (-m,-n)$.
	    \item
	    $\dim M_{(a,b)} = M_{(0,-n)}$ for all $(a,b) \in \mathbb{Z}^2$ with $a \geq 0$ and $b \leq -n$.
	\end{enumerate}
\end{maintheoremf}
Finally, we highlight an interesting application of our results.
\begin{maintheoremg}[Theorem \ref{Bproj cohen}]
\label{thm:maintheoremg}
Let $(A,\mathfrak{m})$ be regular local containing a field of characteristic $p>0$. Let $R=A[x_1,\ldots,x_m,y_1,\ldots,y_n]$ be standard bigraded polynomial ring over $A$. Consider $S=R/I$ where $I$ is a bi-homogeneous ideal of $R$. Assume that $S$ is equidimensional, $\Bproj(S)\neq \phi$ and Cohen-Macaulay. Then, $H^j_I(R)_{(m,n)}=0$ for all $(m,n)\geq (0,0)$ and all $j>\height I$.
\end{maintheoremg}
We now give a brief summary of the contents of this paper. The paper is organized into ten sections. In Section 2, we introduce the necessary definitions and results that will be referenced throughout the paper. Section 3 extends several results by Ma and Zhang from \cite{Ma} to the bigraded setting, largely following their proof techniques. Section 4 is dedicated to proving the vanishing result \hyperref[thm:maintheorema]{Theorem A}, while Section 5 focuses on proving the rigidity results \hyperref[thm:maintheoremb]{Theorem B}  and \hyperref[thm:maintheoremc]{Theorem C}, both when $A$ is an infinite field of characteristic $p>0$. In the next section we prove \hyperref[thm:maintheorema]{Theorem A}, \hyperref[thm:maintheoremb]{Theorem B} and \hyperref[thm:maintheoremc]{Theorem C}. In section 7, we prove \hyperref[thm:maintheoremd]{Theorem D}. In Sections 8 and 9, we prove \hyperref[thm:maintheoreme]{Theorem E} and \hyperref[thm:maintheoremf]{Theorem F}, respectively. Finally, the last section, Section 10, is dedicated to proving \hyperref[thm:maintheoremg]{Theorem G}.

\section{Preliminaries}
In this section, we discuss few preliminary results that we will use throughout this paper. Let $R$ be a bigraded ring. By $\Bmod(R)$, we denote the category of bigraded $R$-modules.
 \s \textbf{Bigraded Lyubeznik functors:}
Let $A$ be a commutative Noetherian ring and $R=A[x_1,\ldots,x_m,y_1,\ldots,y_n]$ be standard bigraded over $A$, i.e., $\bideg(A)=(0,0)$, $\bideg(x_i)=(1,0)$ and  $\bideg(y_j)=(0,1)$ for all $i$ and $j$. We say $Y$ is bi-homogeneous closed subset of $\Spec(R)$ if $Y=V(f_1,\ldots,f_s)$, where each $f_i$ is bi-homogeneous polynomial in $R$.

 We say $Y$ is a bi-homogeneous locally closed subset of $\Spec(R)$ if $Y=Y''-Y'$, where $Y''$ and $Y'$ are bi-homogeneous closed subset of $\Spec(R)$. We have an exact sequence of functors on $\Bmod(R)$ as follows:
 $$H^i_{Y'}(-)\rightarrow H^i_{Y''}(-)\rightarrow H^i_Y(-)\rightarrow H^{i+1}_{Y'}(-)$$
 A bigraded Lyubeznik functor is composite functor of the form $\mathcal{T}=\mathcal{T}_1\circ \mathcal{T}_2\ldots \mathcal{T}_k$ where $\mathcal{T}_j$ is of the form $H^i_{Y_j}$, where $Y_j$ is bi-homogeneous locally closed subset of $\Spec(R)$ or kernal of any row appearing in the above exact sequence with $Y'=Y_j'$ and $Y''=Y_j''$, where $Y_j'\subseteq Y_j''$ are two bi-homogeneous closed subset of $\Spec(R)$.
    \s \textbf{$F$-modules:} For each integer $e\geq 1$, let $^{e}R$ denote the $R$-bimodule that is same as $R$ as a left $R$-module and whose right $R$-module structure is given by $r'r=r^{p^e}r'$
     for all $r'\in\  ^eR$ and $r\in R$. An $F$-module is an $R$-module $M$ equipped with an $R$-module isomorphism $M\xrightarrow{\theta}F_R(M)= \ ^{1}R\otimes_R M$. If $M$ is a bigraded $R$-module then there is a natural grading on $F_R(M)$ defined by $$|r'\otimes m|=|r'|+p\ |m|$$ for homogeneous elements $r'\in R$ and $m\in M$. An $F$-module $(M,\theta)$ is called bigraded $F$-module if $M$ is bigraded and $\theta$ is bidegree preserving.
     
     \s \textbf{$F_R$-finite, $F_R$-module \cite[2.1]{Lyu-Fmod}:} A generating morphism of an $F$-module $M$ is an $R$-module monomorphism $\beta:U\rightarrow F(U)$, where $U$ is some $R$-module, such that $M$ is the limit of the inductive system in the top row of the commutative diagram
     \[\begin{tikzcd}
	U & {F(U)} & {F^2(U)} & \cdots \\
	{F(U)} & {F^2(U)} & {F^3(U)} & \cdots
	\arrow["\beta", from=1-1, to=1-2]
	\arrow["\beta"', from=1-1, to=2-1]
	\arrow["{F(\beta)}", from=1-2, to=1-3]
	\arrow["{F(\beta)}"', from=1-2, to=2-2]
	\arrow["{F^2(\beta)}", from=1-3, to=1-4]
	\arrow["{F^2(\beta)}"', from=1-3, to=2-3]
	\arrow["{F(\beta)}", from=2-1, to=2-2]
	\arrow["{F^2(\beta)}", from=2-2, to=2-3]
	\arrow["{F^3(\beta)}", from=2-3, to=2-4]
\end{tikzcd}\]

     and $\theta:M\rightarrow F(M)$, the structure isomorphism of $M$, is induced by the vertical arrows of this diagram.

     A $F_R$-module $M$ will be called $F_R$-finite, $F_R$-module if $M$ has a generating morphism $\beta:U\rightarrow F(U)$ with $U$ a finitely generated $R$-module. If in addition, $\beta$ is injective, we call $U$ a root of $M$. If $M$ is bigraded and $\beta$ is bidegree preserving, we say that $M$ is a bigraded $F_R$-finite, $F_R$-module.
     \s Let $R$ be a regular ring of characteristic $p>0$. If $\mathcal{T}$ is a Lyubeznik functor on $R$, then $\mathcal{T}(R)$ is $F_R$-finite, $F_R$-module, see \cite[Corollary 2.14]{Lyu-Fmod}. A similar argument proves that if $R$ is bigraded and $\mathcal{T}$ is a bigraded Lyubeznik functor on $\Bmod(R)$ then $\mathcal{T}(R)$ is bigraded $F_R$-finite, $F_R$-module.
     \s\cite[1.3]{Lyu-Fmod} \label{pi take f finite to f finite} Let $R\xrightarrow{\pi}B$ be a ring homomorphism where $B$ is regular. Let $\pi_{*}^{\prime}:\Mod(A)\rightarrow \Mod(B)$ be the functor $B\otimes_R -$. Lyubeznik constructs an isomorphism of functors $\phi:\pi_{*}^{\prime}\circ F_R\rightarrow F_B\circ \pi_{*}^{\prime}$. Then he constructs $\pi^*$ between the categories of $F$-modules over $R$ and $F$-modules over $B$ as follows 
     \begin{align*}
         \pi^*(M,\theta)&=(\pi_{*}^{\prime}(M),\phi\circ \pi_{*}^{\prime}(\theta))\\ 
     \pi^*(f)&=\pi_{*}^{\prime}(f)
     \end{align*}
     In \cite[Proposition 2.9]{Lyu-Fmod}, it is proved that $\pi^*$ takes $F_R$-finite, $F_R$-modules to $F_B$-finite, $F_B$-modules. 
\s \label{Lyu functor exist} If $A\rightarrow B$ is a flat map then we have flat map of bigraded rings $$R=A[x_1,\ldots,x_m,y_1,\ldots,y_n]\rightarrow B[x_1,\ldots,x_m,y_1,\ldots,y_n]$$ and  $\mathcal{T}$ is a bigraded Lyubeznik functor on $\Bmod(R)$ then there is a bigraded Lyubeznik functor $\mathcal{F}$ on $\Bmod(S)$ with $\mathcal{F}(S)=\mathcal{T}(R)\otimes_RS=\mathcal{T}(R)\otimes_AB$.

    We now outline a few results from \cite{TP-koszul}, which will be instrumental in proving various results throughout this paper.
    \begin{theorem}\cite[Theorem 1.1]{TP-koszul}\label{F-finite}
        Let $K$ be an infinite field of characteristic $p$. Let $R$ be one of the following rings 
        \begin{enumerate}[\rm (i)]
            \item $K[y_1,\ldots,y_d]$
            \item  $K[[y_1,\ldots,y_d]]$
            \item  $A[x_1,\ldots,x_m]$ where $A= K[[y_1,\ldots,y_d]]$
        \end{enumerate}
        Let $M$ be a $F_R$-finite, $F_R$-module module. Fix $r\geq 1$. Then the Koszul homology modules $H_i(y_1,\ldots,y_r;M)$ are $F_{\overline{R}}$-finite $F_{\overline{R}}$-modules where $\overline{R}=R/(y_1,\ldots,y_r)$ and for $i=0,\ldots,r$.
    \end{theorem}
An analogous proof extends to the bigraded setting of Lemma 8.1 from \cite{TP-koszul} as follows:
\begin{theorem}\label{ordinal}
Let $A=K[[z_1,\ldots,z_d]]$ where $K$ is an infinite field of characteristic $p$. Let $M=\bigoplus_{i,j} M_{(i,j)}$ be bigraded $F_R$-finite, $F_R$-module  where $R=A[x_1,\ldots, x_m,y_1,\ldots,y_n]$ is standard bigraded. If $M_{(a,b)}$ is supported only at maximal ideal of $A$, then $M_{(a,b)}\cong E^{\alpha}$ where $E$ is injective hull of $K$ as $A$-module and $\alpha$ is an ordinal which may be infinite.
\end{theorem}
We state the following two well-known results concerning associated primes:
\begin{lemma}\label{Ass of M/I}
    Let $A$ be a Noetherian ring, I an ideal of $A$ and let $M$ be an $A$-module. Then,
    \begin{enumerate}[\rm (1)]
        \item $\Ass \Gamma_I(M)=\{P\in \Ass_A M |\  P\supseteq I\}$.
        \item $\Ass M/{\Gamma_I(M)}=\{P\in \Ass_A M |\  P\not\supseteq I\}$. 
    \end{enumerate}
\end{lemma}
\begin{proposition}\label{ass-contraction}
    If $f: A\rightarrow B$ is a homomorphism of Noetherian rings and $M$ is a $B$-module, then the associated primes of $M$ as an $A$-module are contractions of the associated primes as a $B$-module to $A$. More precisely, $\Ass_AM=A\cap \Ass_BM$.
\end{proposition}
\section{Bigraded Eulerian \texorpdfstring{$\D$}{D}-module}
Let $A=K[x_1,\ldots,x_m]$ be a polynomial ring over a field $K$ of characteristic $p >0$. Consider $R=A[y_1,\ldots,y_n]$ with $\bideg(x_i)=(1,0)$ and $\bideg(y_j)=(0,1)$ for all $i,j$. Let $\D$ be the corresponding ring of differential operators on $R$, i.e., $\D=R\left\langle\partial_i^{[j]},\delta_r^{[s]} \mid 1 \leq i \leq m, 1 \leq r \leq n, 1 \leq j, 1\leq s\right\rangle$ where $\partial_i^{[j]}=\frac{1}{j!}\frac{\partial^j}{\partial x_i^j}$ and $\delta_r^{[s]}=\frac{1}{s!}\frac{\delta^s}{\delta y_r^s}$. Clearly $\D$ is bigraded with $\bideg(x_i)$ and $\bideg(y_i)$ same as above and $\bideg(\partial_i^{[j]})=(-j,0)$, $\bideg(\delta_r^{[s]})=(0,-s)$. We define  
$$E^X_r:=\sum_{i_1+i_2+\cdots+i_m=r, i_1 \geq 0, \ldots, i_m \geq 0} x_1^{i_1} \cdots x_m^{i_m} \partial_1^{\left[i_1\right]} \cdots \partial_m^{\left[i_m\right]}$$ $$E^Y_r:=\sum_{r_1+r_2+\cdots+r_n=r, r_1 \geq 0, \ldots, r_n \geq 0} y_1^{r_1} \cdots y_n^{r_n} \delta_1^{\left[r_1\right]} \cdots \delta_n^{\left[r_n\right]}.$$
\begin{definition}
    A bigraded $\mathcal{D}$-module $M$ is called Eulerian if $m\in M$ with $\bideg(m)=(m_1,m_2)$ then we have the followings:

$$E^X_r m=\binom{m_1}{r}m\  \text{and}\   E^Y_r m=\binom{m_2}{r}m$$ for all $r\geq 1$.
\end{definition}
This section is motivated by the work of Ma and Zhang \cite{Ma}. We extend many of their results to the bigraded setting, primarily adapting their proof techniques.
\begin{proposition}
\label{R}
    The polynomial ring $R$ is bigraded Eulerian $\D$-module.
\end{proposition}
\begin{proof}
    It is enough to prove for monomials of the form $x_1^{i_1}x_2^{i_2}\ldots x_m^{i_m}y_1^{r_1}y_2^{r_2}\ldots y_n^{r_n} $. Let this monomial be $f$. Then $\bideg(f)=(i_1+i_2+\ldots+i_m, r_1+r_2+\ldots+r_n)$. Now
    \begin{align*}    
    E^X_r f\ &=\sum_{i_1+i_2+\cdots+i_m=r, i_1 \geq 0, \ldots, i_m \geq 0} x_1^{i_1} \cdots x_m^{i_m} \partial_1^{\left[i_1\right]} \cdots \partial_m^{\left[i_m\right]}(x_1^{i_1}\cdots x_m^{i_m}y_1^{r_1}y_2^{r_2}\ldots y_n^{r_n})\\ &= (y_1^{r_1}\cdots y_n^{r_n}) \sum_{i_1+i_2+\ldots+i_m=r, i_1 \geq 0, \ldots, i_m \geq 0} x_1^{i_1} \cdots x_m^{i_m} \partial_1^{\left[i_1\right]} \cdots \partial_m^{\left[i_m\right]}(x_1^{i_1}\cdots x_m^{i_m})\\&= (y_1^{r_1}y_2^{r_2}\ldots y_n^{r_n})\binom{i_1+i_2+\ldots+i_m}{r} x_1^{i_1}x_2^{i_2}\cdots x_m^{i_m}\\&= \binom{i_1+i_2+\ldots+i_m}{r} f.
    \end{align*}
    Similarly we can show that    
    $
    E^Y_r f= \binom{r_1+r_2+\ldots+r_n}{r} f
   $. Hence $R$ is bigraded Eulerian $\D$-module.
\end{proof}
The following proposition shows that if a module is bigraded Eulerian, this property is preserved when we localize the module at a homogeneous element.
\begin{proposition}
    If $M$ is bigraded Eulerian $\D$-module, so is $S^{-1}M$ for each homogeneous multiplicative system $S\subset R$. In particular $M_f$ is bigraded Eulerian $\D$-module for each homogeneous polynomial $f\in R$.
\end{proposition}
\begin{proof}
 We note that $\partial_i^{[j]}$ is $R^{p^e}$-linear if $p^e \geq j+1$. So we have
$$
\partial_i^{[j]}(m)=\partial_i^{[j]}\left(f^{p^e} \cdot \frac{m}{f^{p^e}}\right)=f^{p^e} \cdot \partial_i^{[j]}\left(\frac{m}{f^{p^e}}\right).
$$

So if $p^e \geq r+1$ and $f \in S$, then $\partial_i^{[j]}\left(\frac{m}{f^{p^e}}\right)=\frac{1}{f^{p^e}} \partial_i^{[j]}(m)$ for every $j \leq r$ in $S^{-1} M$. Therefore,
$$
E^X_r \cdot \frac{m}{f^{p^e}}=\frac{1}{f^{p^e}} E^X_r \cdot m.
$$
By similar argument by replacing $\partial$ by $\delta$ we have 
$$
E^Y_r \cdot \frac{m}{f^{p^e}}=\frac{1}{f^{p^e}} E^Y_r \cdot m.
$$
Take $\frac{m}{f^t} \in S^{-1} M$ where $\bideg(m)=(m_1,m_2)$ and $\bideg(f)=(r_1,r_2)$. Then $\bideg(\frac{m}{f^t})=(m_1-tr_1,m_2-tr_2)$. Multiplying both the numerator and denominator by a large power of $f$, we write $\frac{m}{f^t}=\frac{f^{p^e-t} m}{f^{p^e}}$ for some $p^e \geq \max \{r+1, t\}$. Now
\begin{align*}
    E_r^X \left( \frac{m}{f^t} \right) & = E_r^X \left( \frac{f^{p^e - t} m}{f^{p^e}} \right) \\
    & = \frac{1}{f^{p^e}} E_r^X ( f^{p^e - t} m ).
\end{align*}
 We note that \begin{align*}
     \bideg( f^{p^e - t} m) &=\bideg( f^{p^e - t})+\bideg(m)\\&=p^e\bideg(f)-\bideg(f^t)+\bideg(m)\\&= p^e(r_1,r_2)-t(r_1,r_2)+(m_1,m_2)\\&=(p^er_1-tr_1+m_1,p^er_2-tr_2+m_2).
 \end{align*} 
 Hence
 \begin{align*}
     E_r^X \left( \frac{m}{f^t} \right) &=\frac{1}{f^{p^e}}\binom{p^er_1+m_1-tr_1}{r}f^{p^e - t} m\\&= \binom{m_1-tr_1}{r} \frac{m}{f^t}.
  \end{align*}
   Similarly 
   \begin{align*}
    E_r^Y \left( \frac{m}{f^t} \right) & = E_r^Y \left( \frac{f^{p^e - t} m}{f^{p^e}} \right) \\
    & = \frac{1}{f^{p^e}} E_r^Y ( f^{p^e - t} m )\\&=\frac{1}{f^{p^e}}\binom{p^er_2+m_2-tr_2}{r}f^{p^e - t} m\\&= \binom{m_2-tr_2}{r} \frac{m}{f^t}.
\end{align*}
 \end{proof}
\begin{remark}
If $(M,\theta)$ is an $F$-module then the map $$\alpha_e:M\xrightarrow{\theta}F(M)\xrightarrow{F(\theta)}F^2(M)\xrightarrow{F^2(\theta)}...\rightarrow F^e(M)$$
 induced by $\theta$ is also an isomorphism. This induces a $\D$-module structure on $M$. We specify the action of $\partial_1^{[i_1]}\ldots\partial_m^{[i_m]} $ on $M$. Choose $e$ such that $p^e\geq (i_1+\ldots+i_m)+1$. Given an element $m$, we have $\alpha_e(m)=\sum_i y_i\otimes z_i $ where $y_i\in$ $^eR$ and $z_i\in M$ and we define $$\partial_1^{[i_1]}\ldots\partial_m^{[i_m]} m:=\alpha_e^{-1}\left(\sum \partial_1^{[i_1]}\ldots\partial_m^{[i_m]}y_i\otimes z_i\right).$$
 The action of $\delta_1^{[i_1]}\ldots\delta_m^{[i_m]} $ is defined similarly.
 \end{remark} 

The next result gives a relation between bigraded $F$-module and Eulerian $\D$-module.
\begin{theorem}
\label{d-mod}
    If $M$ is bigraded $F$-module, then $M$ is bigraded Eulerian as $\D$-module.
\end{theorem}
\begin{proof}
    Let $m\in M$ be any homogeneous element of bidegree $(m_1,m_2)$. Pick $e$ such that $p^e\geq r+1$. Since $M$ is a graded $F$-module, we have a bidegree preserving isomorphism $\alpha_e: M\rightarrow F^e_R(M)$. Let $\alpha_e(m)=\sum_i y_i\otimes z_i $ where $y_i\in R$ and $z_i\in M$ are homogeneous. Now $(m_1,m_2)=\bideg(m)=\bideg(\alpha_e(m))=\bideg(y_i\otimes z_i)=p^e\bideg(z_i)+\bideg(y_i)=p^e(a_i,b_i)+(r_i,p_i)=(p^ea_i+r_i,p^eb_i+p_i)$, where $\bideg(z_i)=(a_i,b_i)$ and $\bideg(y_i)=(r_i,p_i)$. Therefore $m_1=p^ea_i+r_i$ and $m_2=p^eb_i+p_i$ for all $i$. Now 
   \begin{align*}
       E_r^Xm &=\alpha_e^{-1}\left(\sum_i(E_r^X y_i)\otimes z_i\right)\\&=\alpha_e^{-1}\left(\sum_i\binom{r_i}{r} y_i\otimes z_i\right)\\&=\alpha_e^{-1}\left(\binom{r_i}{r}\sum_i y_i\otimes z_i\right)\\&=\binom{r_i}{r} m\\&=\binom{p^ea_i+r_i}{r}m\\&=\binom{m_1}{r} m.
   \end{align*} 
   Similarly we have $E_r^Ym=\binom{m_2}{r}m$, and this proves the result.
\end{proof}
Denote by ${ }^* E$, the bigraded injective hull of $R/\mathfrak{m}$. The following result is well known.
\begin{proposition}
\label{iso}
    Let $R=K[x_1,x_2,\ldots,x_m,y_1,y_2,\ldots,y_n]$ as defined above and let $\mathfrak m$ be homogeneous maximal ideal of $R$. Then there exist a graded isomorphism $$\frac{D}{D\mathfrak m}\rightarrow { }^* E \text {. }$$
    \end{proposition}
\begin{remark}
\label{a=b}
    $\binom{a}{r}=\binom{b}{r}$ for all $r\in \mathbb{N}$ if and only if $a=b$ in any characteristic.
\end{remark}
\begin{proposition}
    If $M$ is bigraded Eulerian $\D$-module, then $M(r,s)$ is bigraded Eulerian $\D$-module if and only if $(r,s)=(0,0)$.
\end{proposition}
\begin{proof}
    Let $m\in M(r,s)$ be homogeneous of bidegree $(u,v)$. Since $M$ is bigraded Eulerian $\D$-module and $m\in M(r,s)_{(u,v)}=M_{(u+r,v+s)}$, we get for each $r\geq 1$
    
$$E^X_r m=\binom{u+r}{r}m\  \text{and}\   E^Y_r m=\binom{v+s}{r}m.$$
$M(r,s)$ is bigraded Eulerian $\D$-module  if and only if
$$E^X_r m=\binom{u}{r}m\  \text{and}\   E^Y_r m=\binom{v}{r}m.$$ 
Therefore, $\binom{u+r}{r}=\binom{u}{r}$ and $\binom{v+s}{r}=\binom{v}{r}$ for all $r\geq 1$. So by Remark \ref{a=b}, $r=0$ and $s=0$.
\end{proof}
\begin{proposition}
\label{big}
    The bigraded Eulerian $\D$-module ${ }^* E(a,b)$ is Eulerian bigraded if and only if $(a,b)=(m,n)$.
\end{proposition}
\begin{proof}
    It is enough to show that ${ }^* E(m,n)$ is bigraded Eulerian $\D$-module since ${ }^* E(a,b)={ }^* E(m,n){(a-m,b-n)}$. As a $K$-vector space it is known that ${ }^* E(m,n)$ is spanned by $x_1^{i_1}x_2^{i_2}\ldots x_m^{i_m}y_1^{r_1}y_2^{r_2}\ldots y_n^{r_n}$ with each powers less than or equal to $-1$. By computation of Proposition \ref{R}, it is clear that ${ }^* E(m,n)$ is Eulerian bigraded because in ${ }^* E(m,n)$ it has bidegree $(i_1+\ldots+i_m,r_1+\ldots+r_n)$.
\end{proof}
The following theorem is due to L. Ma and W. Zhang.
\begin{theorem}
    Let $M$ be a Eulerian graded $\D$-module. If $\Supp_R(M)=\{\mathfrak m\}$, then as a graded $\D$-module $M$ is isomorphic to direct sum of copies of $^*E(n)$. 
\end{theorem}
In the above theorem, $R=K[x_1,\ldots,x_n]$, where $K$ is a field of arbitrary characteristic. $\D$ is corresponding ring of differential operators together with $\deg(x_i)$=1 for all $i$. In fact, a similar result holds true in the bigraded context, which we state below and can be proven in a manner analogous to the main proof.

\begin{theorem}
\label{Supp}
 Let $M$ be a bigraded $\D$-module. If Supp$(M)=\{\mathfrak{ m}\}$, then as a bigraded $\D$-module $M \cong\bigoplus_{i,j} \frac{D}{D \mathfrak{m}}(p_i,q_j) \cong \bigoplus_{i,j}{ }^* E(p_i,q_j)$.
 \end{theorem}
Now, we observe an easy corollary of the previously stated theorem.
\begin{corollary}
 Let $M$ be a bigraded Eulerian $\D$-module. If Supp$(M)=\{\mathfrak{ m}\}$, then as a bigraded $\D$-module $M$ is isomorphic to direct sum of copies of ${ }^* \mathrm{E}(m,n)$.
 \end{corollary}
\begin{proof}
    By Theorem \ref{Supp}, $M$ is isomorphic to $\bigoplus_{i,j}{ }^* E(p_i,q_j)$ as graded $\D$-module. But since $M$ is Eulerian so is  $\bigoplus_{i,j}{ }^* E(p_i,q_j)$. By Proposition \ref{big}, $p_i=m$ and $q_j=n$ for all $i,j$. This proves the result. 
\end{proof}
\section{Vanishing properties }
This section focuses on proving an important result regarding the vanishing of a bigraded $F$-finite, $F$-module over $R=K[x_1,\ldots, x_m,y_1,\ldots,y_n]$ under certain condition (See Theorem \ref{Vanishing of M}).
\begin{theorem}\cite[Lucas Theorem]{Lucas} \label{Lucas th}
    Let $p$ be a prime number and $n$ and $m$ are positive integers. Let $n=\sum a_ip^i$ and $m=\sum b_jp^j$ be $p$-adic representations of $n$ and $m$ respectively. Then $\binom{n}{m}\equiv \prod_i \binom{a_i}{b_i}(\bmod p)$.
\end{theorem}
\begin{lemma}\label{p-adic lemma}
    Assume that $t\geq 1$ and $p\nmid t$. Then coefficient of $p^e$ in the $p$-adic representation of $p^et-1$ can not be $p-1$.
\end{lemma}
\begin{proof}
    We note that $p^et-1=p(p^{e-1}t-1)+(p-1)$. Therefore the coefficient of $p^e$ in the $p$-adic representation of $(p^et-1)$ is the remainder upon dividing $t-1$ by $p$. If $p\mid (t-1)$, then the remainder is $0$. Now suppose $p\nmid (t-1)$. So either $p> (t-1)$ or $p<(t-1)$. In the first case we again get the remainder $(t-1)$ which is not equal to $(p-1)$ as $p\nmid t$. Hence assume that $p<(t-1)$ and the remainder is $(p-1)$. Therefore, $t-1=pk+(p-1)$ and this implies that $p\mid t$, a contradiction. 
\end{proof}
This lemma allows us to prove the following result.
\begin{lemma}
\label{p does not divide}
    If $p$ is a prime number and $t\in\mathbb{Z}$ with $p\nmid t$, then $p\nmid \left[\binom{P^et-1}{p^e}+1\right]$.
\end{lemma}
\begin{proof}
  Assume that $t\geq 1$ and $p\mid \left[\binom{P^et-1}{p^e}+1\right]$. Let $p^et-1=\sum a_ip^i$ be the $p$-adic representation of $p^et-1$. By Theorem \ref{Lucas th}, we have \begin{align*}
      \binom{p^et-1}{p^e}\equiv a_e(\bmod p)
      \implies 1+\binom{p^et-1}{p^e}\equiv 1+a_e(\bmod p).
  \end{align*} 
  Since $p\mid \left[\binom{P^et-1}{p^e}+1\right]$, we have that $1+a_e\equiv 0(\bmod p)$, a contradiction by Lemma \ref{p-adic lemma}.

Now assume that $t$ is negative integer and $p\mid \left[\binom{p^et-1}{p^t}+1\right]$. Note that $\binom{p^et-1}{p^t}=(-1)^{p^i}\binom{p^e-p^et}{p^e}=(-1)^{p^i}\binom{p^ek}{p^e}$ where $k=1-t$. Suppose $k\leq (p-1)$ and let $p$ is odd. Then again by Theorem \ref{Lucas th}, $\binom{p^ek}{p^k}\equiv k(\bmod p)\implies p\mid (1-k)=t$, contradiction. If $p=2$, then we get $2\mid 1+k=2-t\implies 2\mid t$, a contradiction.

Finally assume that $k\geq p$ and $p$ is odd. Let $k=\sum r_ip^i $ be $p$-adic decomposition of $k$. Then $p^ek=r_op^e+r_1p^{e+1}+\ldots$ is $p$-adic decomposition of $p^ek$. By Theorem \ref{Lucas th}, $\binom{p^ek}{p^e}\equiv r_0(\bmod p)\implies p\mid (1-r_0) $, a contradiction when $r_0\neq 1$ since $0\leq r_0\leq p-1$. If $r_0=1$, then $p\mid (1-k)=t$ which is again not true. Now let $p=2$ , then $2\mid (1+r_0)$. Here $r_0$ is coefficient of $2^0$ in the $2$-adic representation of $k$, which can not be $1$. Suppose it is $1$. Then $k=1-t=2s+ 1$ for some $s\in \mathbb{Z}$ $\implies 2\mid t$, a contradiction. So only possibility is $r_0=0$. This is again not possible as $2\mid (1+r_0)$.
\end{proof}
As a consequence of Lemma \ref{p does not divide}, we have the following result, which will be used to prove several other results.
\begin{proposition}
\label{Koszul Vanish for m=1,n=1}
    Let $R=K[x,y]$ be standard bigraded over a field of characteristic $p>0$. Assume that $M=\bigoplus M_{(u,v)}$ is a bigraded $F_R$-finite, $F_R$-module. Then, $H_i(x,M)_{(u,v)}=0$ for $u\neq 0$ and for $i=0,1$. Also $H_i(y,M)_{(u,v)}=0$ for $v\neq 0$ and for $i=0,1$.
\end{proposition}
\begin{proof}
    Let $u\neq 0$ and  $\xi\in M_{(u,v)}$. Assume $u=p^et$ where $p$ does not divide $t$. Since $M$ is bigraded Eulerian $$E^X_{p^e}\xi=x^{p^e}\partial^{[p^e]}\xi=\binom{p^et}{p^e}\xi.$$
    It can be easily verified that $\binom{p^et}{p^e}$ is not divisible by $p$. It follows that $\xi\in xM$. Thus $H_0(x,M)_{(u,v)}=0$.
    
    Now let $\xi\in H_1(x,M)_{(u,v)}\subseteq M_{(u-1,v)}$. Since $M$ is Eulerian so
    \begin{align*}
        x\partial \ \xi=(u-1)\xi &\implies (\partial x-1)\xi=(u-1)\xi\\& \implies u\xi=0\\&\implies \xi=0 \ \text{if}\  p\nmid u.
    \end{align*}
     If $p\mid u$, then let $u=p^et$ where $p\nmid t$. It is clear that $p\mid \binom{p^e}{i}$ for $i\neq 0$ and $i\neq p^e$. We have the following formula;
     \begin{align*}  
     \partial^{[p^e]}x^{p^e}&=\sum_{i=0}^{p^e}\binom{p^e}{i}x^{p^e-i}\partial^{[p^e-i]}\\&=x^{p^e}\partial^{[p^e]}+1.
     \end{align*}
Therefore,
$$0=\partial^{[p^e]}x^{p^e}\xi=x^{p^e}\partial^{[p^e]}\xi+\xi.$$
Since $M$ is Eulerian and $\xi\in M_{(p^et-1,v)}$ we get
$$\left[\binom{p^et-1}{p^e}+1\right]\xi=0\implies \xi=0.$$ since $p\nmid \left[\binom{p^et-1}{p^e}+1\right]$ by Lemma \ref{p does not divide}.

Similarly we can prove $H_i(y,M)_{(u,v)}=0$ for $v\neq 0$ and for $i=0,1$.
\end{proof}
Similar proof of Proposition \ref{Koszul Vanish for m=1,n=1} gives us the following.
\begin{proposition}\label{Koszul Vanish for (m,1),(1,n)}
 Let $K$ be a field of characteristic $p>0$ and $R=K[x,y_1,\ldots,y_n]$ be standard bigraded, i.e., $\bideg(K)=(0,0)$, $\bideg(x_i)=(1,0)$ and  $\bideg(y_j)=(0,1)$ for all $i$ and $j$. Assume that $M=\bigoplus M_{(u,v)}$ is a bigraded $F_R$-finite, $F_R$-module. Then, $H_i(x,M)_{(u,v)}=0$ for $u\neq 0$ and for $i=0,1$.

  Similarly, if $R=K[x_1,\ldots,x_m,y]$ and  $M=\bigoplus M_{(u,v)}$ is a bigraded $F_R$-finite, $F_R$-module then $H_i(y,M)_{(u,v)}=0$ for $v\neq 0$ and for $i=0,1$.
\end{proposition}
\begin{remark}
\label{M=0}
  Let $M=\bigoplus_{n\in \mathbb{Z}} M_n$ be a graded $D(K[x])$-module and $M_n=0$ for $|n|\gg 0$. If $x$ is $M$-regular then we assert that $M=0$. Since $x$ is $M$-regular so $H_1(x,M)=0$. Let $M=\bigoplus_{n=r}^sM_n$ with $M_s\neq 0$. Since $(x)M_s=0$, we get $M_s\subseteq H_1(x,M)$. The later module is zero so $M_s=0$, a contradiction. Thus $M=0$.
 \end{remark}
  The following observation will be useful.
\begin{lemma}
\label{M=01}
     Let $K$ be a field of characteristic $p>0$ and $M=\bigoplus_{i,j} M_{(i,j)}$ be bigraded $F_R$-finite, $F_R$-module where $R=K[x_1,\ldots, x_m,y_1]$ with $\bideg(x_i)=(1,0)$ and $\bideg(y_1)=(0,1)$ for all $i$. Let $M_{(m,n)}=0$ for all $|m|\gg 0 $ and $|n|\gg 0$. If some $x_i$ is $M$-regular, then $M=0$.
     \end{lemma}
     \begin{proof}
Let $M_{(u,v)}=0$ for $(u,v)\notin S$ where $S$ is the shaded region of the figure 1. Consider $N=\bigoplus_{j\in \mathbb{Z}}M_{(j,r)}$ where $|r|>a$. For a fixed $j$ choose a sufficiently large positive integer $s_j$ such that $j+s_j>b$. Consider the following exact sequence $$0\rightarrow M_{(j,r)}\xrightarrow{x_i^{s_j}} M_{(j+s_j,r)}\rightarrow \overline{M}_{(j+s_j,r)}\rightarrow 0.$$
Therefore $M_{(j,r)}=0$ as $M_{(j+s_j,r)}=0$. Hence $N=0$. 
\begin{center}
\begin{tikzpicture}[scale=0.2]
\draw[->, ultra thick] (-7.5,0)--(8.5,0) node[right]{$x$};
\draw[->, ultra thick] (0,-7.5)--(0,8.5) node[above]{$y$};
\draw[line width=0.25mm,red](-1.5,-7.25)-- (-1.5,7.25)node[left]{$x=-b$};
\draw[line width=0.25mm,red](2,-7.25)-- (2,7.25) node[right]{$x=b$};
\draw[line width=0.25mm,red](-7.5,-1.5)-- (7.5,-1.5) node[below]{$y=-a$};
\draw[line width=0.25mm,red]  (-7.5,2)-- (7.5,2) node[above]{$y=a$};
\draw[fill=brown,fill opacity=0.35,draw=none] (-1.5,7.25)--(-1.5,-7.5)-- (2,-7.5) --(2,7.25);
\draw[fill=brown,fill opacity=0.35,draw=none] (7.5,-1.5)--(-7.5,-1.5)--(-7.5,2)-- (7.5,2);
\node[draw=none] at (-3,1) {{\small $S$}};
\node at (0,-11) {\textit{Figure $1$}};
\end{tikzpicture}	
\hspace{1cm}
\begin{tikzpicture}[scale=0.2]
\draw[->, ultra thick] (-7.5,0)--(8.5,0) node[right]{$x$};
\draw[->, ultra thick] (0,-7.5)--(0,8.5) node[above]{$y$};
\draw[line width=0.25mm,red](-7.5,-1.5)-- (7.5,-1.5) node[below]{$y=-a$};
\draw[line width=0.25mm,red]  (-7.5,2)-- (7.5,2) node[above]{$y=a$};
\draw[fill=brown,fill opacity=0.5,draw=none] (7.5,-1.5)--(-7.5,-1.5)--(-7.5,2)-- (7.5,2);
\node[draw=none] at (-3,1) {{\small $S'$}};
\node at (0,-11) {\textit{Figure $2$}};
\end{tikzpicture}
\end{center}

Now we are in the situation of figure 2. Now consider $N'=\bigoplus_{j\in \mathbb{Z}}M_{(i,j)}$ where $i$ is fixed. It is graded Eulerian $\D(K[y_1])$-module.  If $H^0_{(y_1)}(N')\neq 0$, then it is direct sum of $^*E(1)$ where $^*E$ is injective hull of $K$ as $K[y_1]$-module. But this would imply  $H^0_{(y_1)}(N')_j\neq 0$ for $j\ll 0$ which is a contradiction since $N_j=0$ for $|j|>a$. Hence  $H^0_{(y_1)}(N')= 0$ and this implies $y_1$ is $N'$-regular. By Remark \ref{M=0}, we conclude $N'=0$. So $M=0$.
\end{proof}
Having established the necessary preparatory results, we now state and prove the main result of this section.
\begin{theorem}
\label{Vanishing of M}
   Let $K$ be an infinite field of characteristic $p>0$. Let  $R=K[x_1,\ldots,x_m,y_1,\ldots,y_n]$ be standard bigraded polynomial ring over $K$, that is, $\bideg(K)=(0,0)$, $\bideg(x_i)=(1,0)$ and  $\bideg(y_j)=(0,1)$ for all $i$ and $j$. Let $M=\bigoplus_{i,j} M_{(i,j)}$ be a bigraded $F_R$-finite, $F_R$-module. If $M_{(m,n)}=0$ for all $|m|\gg 0 $ and $|n|\gg 0$, then $M=0$.
\end{theorem}
\begin{proof}
First we prove it for $m=1$ and $n=1$.
\begin{center}
\begin{tikzpicture}[scale=0.2]
\draw[->, ultra thick] (-7.5,0)--(8.5,0) node[right]{$x$};
\draw[->, ultra thick] (0,-7.5)--(0,8.5) node[above]{$y$};
\draw[line width=0.25mm,red](-1.5,-7.25)-- (-1.5,7.25)node[left]{$x=-b$};
\draw[line width=0.25mm,red](2,-7.25)-- (2,7.25) node[right]{$x=b$};
\draw[line width=0.25mm,red](-7.5,-1.5)-- (7.5,-1.5) node[below]{$y=-a$};
\draw[line width=0.25mm,red]  (-7.5,2)-- (7.5,2) node[above]{$y=a$};
\draw[fill=brown,fill opacity=0.35,draw=none] (-1.5,7.25)--(-1.5,-7.5)-- (2,-7.5) --(2,7.25);
\draw[fill=brown,fill opacity=0.35,draw=none] (7.5,-1.5)--(-7.5,-1.5)--(-7.5,2)-- (7.5,2);
\node[draw=none] at (-3,1) {{\small $S$}};
\node at (0,-11) {\textit{Figure $1$}};
\end{tikzpicture}	
\hspace{1cm}

\end{center}

Let $M_{(u,v)}=0$ for $(u,v)\notin S$ where $S$ is the shaded region of the figure 1. Since $M$ is bigraded $F$-module so it is bigraded Eulerian $\D(K[x_1,y_1])$-module by Theorem \ref{d-mod}. Consider $N=\bigoplus_{j\in \mathbb{Z}}M_{(j,r)}$ where $|r|>a$. Since $M$ is bigraded Eulerian as $\D(K[x_1,y_1])$-module, $N$ is $\mathbb{Z}-$graded Eulerian $\D(K[x_1])$-module. If $H^0_{(x_1)}(N)\neq 0$, then it is direct sum of $^*E(1)$ where $^*E$ is injective hull of $K$ as $K[x_1]$-module. But this would imply  $H^0_{(x_1)}(N)_j\neq 0$ for $j\ll 0$ which is a contradiction since $N_j=0$ for $|j|>b$. Hence  $H^0_{(x_1)}(N)= 0$ and this implies $x_1$ is $N$-regular. By Remark \ref{M=0}, we conclude $N=0$.
\begin{center}
\begin{tikzpicture}[scale=0.2]\label{fig2}
\draw[->, ultra thick] (-7.5,0)--(8.5,0) node[right]{$x$};
\draw[->, ultra thick] (0,-7.5)--(0,8.5) node[above]{$y$};
\draw[line width=0.25mm,red](-7.5,-1.5)-- (7.5,-1.5) node[below]{$y=-a$};
\draw[line width=0.25mm,red]  (-7.5,2)-- (7.5,2) node[above]{$y=a$};
\draw[fill=brown,fill opacity=0.5,draw=none] (7.5,-1.5)--(-7.5,-1.5)--(-7.5,2)-- (7.5,2);
\node[draw=none] at (-3,1) {{\small $S'$}};
\node at (0,-11) {\textit{Figure $2$}};
\end{tikzpicture}
\end{center}
Let $S'$ be the new shaded region as in figure 2. We have proved that $M_{(u,v)}=0$ if $(u,v)\notin S'$. Now consider $N'=\bigoplus_{j\in \mathbb{Z}}M_{(i,j)}$ where $i$ is fixed. Considering $N'$ as $\mathbb{Z}$-graded Eulerian $\D(K[y_1])$-module, similar argument as above proves that $N'=0$. Therefore $M=0$.

Now we consider the case $m\geq 2$ and $n=1$. Let $R=K[x_1,\ldots,x_m,y_1]$ with $m\geq 2$. We claim that $H_1(x_1,M)=0$. If $m=2$, then $H_1(x_1,M)\subseteq M(-1,0)$ is $F$-finite, $F$-module over $K[x_2,y_1]$ by Theorem \ref{F-finite}. Since $M_{(m,n)}=0$ for all $|m|\gg 0 $ and $|n|\gg 0$ so $H_1{(x_1,M)}_{(u,v)}=0$ for all $|u|\gg 0 $ and $|v|\gg 0$. Therefore by the base case $H_1(x_1,M)=0$. Now assume it for $m-1$ and we prove this for $m$ where $m\geq 2$. We note that $H_1(x_2,H_1(x_1,M))\subseteq H_1(x_1,M)(-1,0)\subseteq M{(-2,0)}$ and \ref{F-finite} implies that $H_1(x_1,M)$ is $F$-finite, $F$-module over $K[x_2,\ldots,x_m,y_2]$. Again since $M_{(m,n)}=0$ for all $|m|\gg 0 $ and $|n|\gg 0$ so $H_1(x_2,H_1(x_1,M))_{(u,v)}=0$ for all $|u|\gg 0 $ and $|v|\gg 0$. By the induction hypothesis $H_1(x_2,H_1(x_1,M))=0$, so $x_2$ is $H_1(x_1,M)$-regular and therefore $H_1(x_1,M)=0$ by Lemma \ref{M=01}. Again since $x_1$ is $M$-regular, $M=0$ by Lemma \ref{M=01}.

A similar argument proves the result for $K[x_1,y_1,\ldots,y_n]$.

Now consider  $R=K[x_1,\ldots,x_m,y_1,\ldots,y_n]$ and assume that the result holds for $(m,n-1)$ and $(m-1,n)$. Again by Theorem \ref{F-finite}, we get that $H_1(x_1,M)$ is $F_{\overline{R}}$-finite and $H_1(y_1,M)$ is $F_{\overline{S}}$-finite where $\overline{R}=K[x_2,\ldots,x_m,y_1,\ldots,y_n]$  and  $\overline{S}=K[x_1,\ldots,x_n,y_2,\ldots,y_n]$ . By induction $H_1(x_1,M)=0$ and  $H_1(y_1,M)=0$. So by same argument as in Lemma \ref{M=01}, $M_{(i,j)}=0$ for all $i,j$ and hence $M=0$.
\end{proof}
\section{Rigidity properties when \texorpdfstring{$R$}{R} is polynomial ring over \texorpdfstring{$K$}{K}}
Let us denote $\underline{u}=(u_1, \dots, u_r)\in \mathbb{Z}^r$ and $\underline{v}=(v_1,\ldots,v_r)\in \mathbb{Z}^r$. We say that $\underline{v}\geq \underline{u}$ if $v_i\geq u_i$ for all $i= 1, \ldots, r$.

In this section, we study the rigidity properties of bigraded components of $F$-finite, $F$-modules over $R=K[x_1,\ldots, x_m,y_1,\ldots,y_n]$ where $K$ is an infinite field of characteristic $p>0$. The following result is helpful for proving  rigidity results in this section.
\begin{theorem}\label{Koszul vanish for underline x and y}
Let $K$ be an infinite field of characteristic $p>0$. Assume that $M$ is a bigraded $F_R$-finite, $F_R$-module where  $R=K[x_1,\ldots,x_m,y_1,\ldots,y_n]$ is standard bigraded, i.e., $\bideg(K)=(0,0)$, $\bideg(x_i)=(1,0)$ and  $\bideg(y_j)=(0,1)$ for all $i$ and $j$. Let $\underline{x}=x_1,\ldots,x_m$ and $\underline{y}=y_1,\ldots,y_n$. Then $H_i(\underline{x},M)_{(u,v)}=0$ for $u\neq 0$ and for $i=m$ and $i=0$. Also, $H_i(\underline{y},M)_{(u,v)}=0$ for $v\neq 0$ and for $i=n$ and $i=0$.
\begin{proof}
    First we prove that $H_m(\underline{x},M)_{(u,v)}=0$ for $u\neq 0$. We prove this by induction on $m$. If $m=1$, then the result follows from Proposition \ref{Koszul Vanish for (m,1),(1,n)}. Assume this for all values less than equal to $m-1$ and prove this for $m$. We note that $$H_m(\underline{x},M)_{(u,v)}=H_1(x_m,H_{m-1}(x_1,\ldots,x_{m-1},M))_{(u,v)}.$$
    Since by Theorem \ref{F-finite} we get that $H_{m-1}(x_1,\ldots,x_{m-1},M)$ is $F$-finite, $F$-module over $K[x_m,y_1,\ldots,y_n]$, the result follows by induction.

    Now we prove that $H_0(\underline{x},M)_{(u,v)}=0$ for $u\neq 0$. $m=1$ case follows from  \ref{Koszul Vanish for (m,1),(1,n)}.  Assume this for all the values less than equal to $m-1$ and prove this for $m$. Note that $$H_0(\underline{x},M)=H_0(x_m,H_0(x_1,\ldots,x_{m-1},M)).$$
    Again since Theorem \ref{F-finite} implies that $H_{0}(x_1,\ldots,x_{m-1},M)$ is $F$-finite, $F$-module over $K[x_m,y_1,\ldots,y_n]$, the result follows by induction.

     Similarly we can prove that $H_i(\underline{y},M)_{(u,v)}=0$ for $v\neq 0$ and for $i=n$ and $i=0$.
\end{proof}
\end{theorem}
The first result we present is as follows:
\begin{theorem}\label{Rigidity for K-1}
Let $K$ be an infinite field of characteristic $p>0$. Assume that $M=\bigoplus_{i,j} M_{(i,j)}$  is bigraded $F_R$-finite, $F_R$-module where  $R=K[x_1,\ldots,x_m,y_1,\ldots,y_n]$ is standard bigraded, i.e., $\bideg(K)=(0,0)$, $\bideg(x_i)=(1,0)$ and  $\bideg(y_j)=(0,1)$ for all $i$ and $j$. Then,
 \begin{enumerate}[\rm (1)]
     \item The following statements are equivalent:
     \begin{enumerate}[\rm (a)]
     \item $M_{(a,b)}\neq 0$ for all $(a,b)\geq (0,0)$.
         \item There exists $(a_0,b_0)\geq (0,0)$ such that $M_{(a,b)}\neq 0$ for all $(a,b)\geq (a_0,b_0)$.
         \item $M_{(a_1,b_1)}\neq 0$ for some $(a_1,b_1)\geq (0,0)$.
     \end{enumerate}
     \item The following statements are equivalent:
     \begin{enumerate}[\rm (a)]
         \item $M_{(a,b)}\neq 0$ for all $(a,b)\in\mathbb{Z}^2$ with $a\leq -m$ and $b\geq 0$.
         \item There exists $(a_0,b_0)\in \mathbb{Z}^2$ with $a_0\leq -m$ and $b_0\geq 0$ such that $M_{(a,b)}\neq 0$ for all $(a,b)\in \mathbb{Z}^2$ with $a\leq a_0$ and $b\geq b_0$.
         \item $M_{(a_1,b_1)}\neq 0$ for some $(a_1,b_1)\in \mathbb{Z}^2$ with $a_1\leq -m$ and $b_1\geq 0$.
     \end{enumerate}
     \item The following statements are equivalent:
     \begin{enumerate}[\rm (a)]
         \item $M_{(a,b)}\neq 0$ for all $(a,b)\in\mathbb{Z}^2$ with $(a,b)\leq (-m,-n)$.
         \item There exists $(a_0,b_0)\in \mathbb{Z}^2$ with $(a_0,b_0)\leq (-m,-n)$  such that $M_{(a,b)}\neq 0$ for all $(a,b)\leq (a_0,b_0)$.
         \item $M_{(a_1,b_1)}\neq 0$ for some $(a_1,b_1)\in \mathbb{Z}^2$ with $(a_1,b_1)\leq (-m,-n)$.
     \end{enumerate}
\item The following statements are equivalent:
     \begin{enumerate}[\rm (a)]
         \item $M_{(a,b)}\neq 0$ for all $(a,b)\in\mathbb{Z}^2$ with $a\geq 0$ and $b\leq -n$.
         \item There exists $(a_0,b_0)\in \mathbb{Z}^2$ with $a_0\geq 0$ and $b_0\leq -n$ such that $M_{(a,b)}\neq 0$ for all $(a,b)\in \mathbb{Z}^2$ with $a\geq a_0$ and $b\leq b_0$.
         \item $M_{(a_1,b_1)}\neq 0$ for some $(a_1,b_1)\in \mathbb{Z}^2$ with $a_1\geq 0$ and $b_1\leq -n$.
     \end{enumerate}
 \end{enumerate}
\end{theorem}
\begin{proof}
The numbering of the figures corresponds to the parts of the theorem stated. 
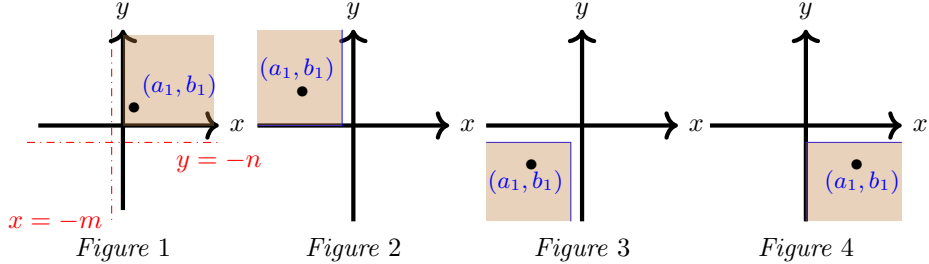
\begin{center}
	\begin{tikzpicture}[scale=0.15]
	\draw[->, ultra thick] (-7.5,0)--(8.5,0) node[right]{$x$};
	\draw[->, ultra thick] (0,-7.5)--(0,8.5) node[above]{$y$};
    \draw[dashdotted, red] (-1,8.5)--(-1,-8.5) node[left]{$x=-m$};
		\draw[dashdotted, red] (-8.5,-1.5)--(8.5, -1.5) node[below]{$y=-n$};
	\draw[fill=brown,fill opacity=0.35,draw=none] (8,8)--(8,0)--(0,0)--(0,8);
     \node[blue,draw=none] at (5,3.5) {\small $(a_1,b_1)$};
     \node at (1,1.5){$\bullet$};
	\node at (0,-11) {\textit{Figure $1$}};
	\end{tikzpicture}
    \hspace{-.5cm} 
\begin{tikzpicture}[scale=0.15]
	\draw[->, ultra thick] (-8.5,0)--(8.5,0) node[right]{$x$};
	\draw[->, ultra thick] (0,-8.5)--(0,8.5) node[above]{$y$};
	\draw[-,blue](-1,8.5)--(-1,0)-- (-8.5,0);
	\draw[fill=brown,fill opacity=0.35,draw=none] (-1,8.5)--(-1,0)--(-8.5,0)--(-8.5,8.5)--(-1,8.5);
	\node[blue,draw=none] at (-5,5) {\small $(a_1,b_1)$};
	\node at (-4.5,3){$\bullet$};
	\node at (0,-11) {\textit{Figure $2$}};
	\end{tikzpicture}
     \hspace{-.3cm}
     \begin{tikzpicture}[scale=0.15]
	\draw[->, ultra thick] (-8.5,0)--(8.5,0) node[right]{$x$};
	\draw[->, ultra thick] (0,-8.5)--(0,8.5) node[above]{$y$};
	\draw[-,blue](-1,-8.5)--(-1,-1.5)-- (-8.5,-1.5);
	\draw[fill=brown,fill opacity=0.35,draw=none] (-1,-8.5)--(-1,-1.5)--(-8.5,-1.5)--(-8.5,-8.5)--(-1,-8.5);
	\node[blue,draw=none] at (-5,-5) {\small $(a_1,b_1)$};
	\node at (-4.5,-3.5){$\bullet$};
	\node at (0,-11) {\textit{Figure $3$}};
	\end{tikzpicture}
    \hspace{-.3cm}
    \begin{tikzpicture}[scale=0.15]
	\draw[->, ultra thick] (-8.5,0)--(8.5,0) node[right]{$x$};
	\draw[->, ultra thick] (0,-8.5)--(0,8.5) node[above]{$y$};
	\draw[-,blue](0,-8.5)--(0,-1.5)-- (8.5,-1.5);
	\draw[fill=brown,fill opacity=0.35,draw=none] (0,-8.5)--(0,-1.5)--(8.5,-1.5)--(8.5,-8.5)--(0,-8.5);
	\node[blue,draw=none] at (5,-5) {\small $(a_1,b_1)$};
	\node at (4.5,-3.5){$\bullet$};
	\node at (0,-11) {\textit{Figure $4$}};
	\end{tikzpicture}	
 \end{center}

(1) We only need to prove that $(c)\implies (a)$ since $(a)\implies (b)\implies (c)$. 
Consider the following exact sequence 
$$0\rightarrow H_m(\underline{x},M)_{(a+m,b)}\rightarrow M_{(a,b)}\rightarrow M^m_{(a+1,b)}.$$
 By Theorem \ref{Koszul vanish for underline x and y}, $H_m(\underline{x},M)_{(a+m,b)}=0$ whenever $a+m\neq 0$. Hence $M_{(a,b)}\neq 0$  implies  $M_{(a+1,b)}\neq 0$ if $a\neq -m$. Therefore if $a_1=0$, then $M_{(i,b_1)}\neq 0$ for $i\geq 0$ as $a_1+t\neq -m$ for all $t\geq 0$. 
  If $a_1\neq 0$, then $M_{(i,b_1)}\neq 0$ for $i\geq a_1$ as $a_1+t\neq -m$ for all $t\geq 0$.
  
  Now consider the following exact sequence $$M^m_{(a-1,b)}\rightarrow M_{(a,b)}\rightarrow H_0(\underline{x},M)_{(a,b)}\rightarrow 0.$$
 By Theorem \ref{Koszul vanish for underline x and y}, $H_0(\underline{x},M)_{(a,b)}=0$ whenever $a\neq 0$. Hence $M_{(a,b)}\neq 0$ implies  $M_{(a-1,b)}\neq 0$ if $a\neq 0$. This implies that $M_{(i,b_1)}\neq 0$ for $0\leq i<a_1$ as $a_1-t\neq 0$ when $0\leq t\leq a_1-1$. Therefore $M_{(i,b_1)}\neq 0$ for all $i\geq 0$.

 Now fix $r\geq 0$, then $M_{(r,b_1)}\neq 0$. Consider the following exact sequence 
$$0\rightarrow H_n(\underline{y},M)_{(a,b+n)}\rightarrow M_{(a,b)}\rightarrow M^n_{(a,b+1)}.$$
 By Theorem \ref{Koszul vanish for underline x and y}, $H_n(\underline{y},M)_{(a, b+n)}=0$ whenever $b+n\neq 0$. Hence $M_{(a,b)}\neq 0$ implies $M_{(a,b+1)}\neq 0$ if $b\neq -n$. Therefore if $b_1=0$, then $M_{(r,j)}\neq 0$ for $j\geq 0$ as $b_1+t\neq -n$ for all $t\geq 0$. 
 If $b_1\neq 0$, then $M_{(r,j)}\neq 0$ for $j\geq b_1$ as $b_1+t\neq -n$ for all $t\geq 0$. 
 
 Now consider the following exact sequence $$M^n_{(a,b-1)}\rightarrow M_{(a,b)}\rightarrow H_0(\underline{y},M)_{(a,b)}\rightarrow 0.$$
 By Theorem \ref{Koszul vanish for underline x and y}, $H_0(\underline{y},M)_{(a,b)}=0$ whenever $b\neq 0$. Hence $M_{(a,b)}\neq 0$ implies $M_{(a,b-1)}\neq 0$ if $b\neq 0$. This implies that $M_{(r,j)}\neq 0$ for $0\leq j<b_1$ as $b_1-t\neq 0$ when $0\leq t\leq b_1-1$. Therefore $M_{(r,j)}\neq 0$ for all $j\geq 0$. Now varying $r$ from $0$ to $\infty$, we get the result.

(2) We only need to prove that $(c)\implies (a)$ since $(a)\implies (b)\implies (c)$. Let $M_{(a_1,b_1)}\neq 0$ for some $(a_1,b_1)\in \mathbb{Z}^2$ with $a_1\leq -m$ and $b_1\geq 0$. Now consider the following exact sequence $$M^m_{(a-1,b)}\rightarrow M_{(a,b)}\rightarrow H_0(\underline{x},M)_{(a,b)}\rightarrow 0.$$
 By Theorem \ref{Koszul vanish for underline x and y}, $H_0(\underline{x},M)_{(a,b)}=0$ whenever $a\neq 0$. Hence $M_{(a,b)}\neq 0$ implies  $M_{(a-1,b)}\neq 0$ if $a\neq 0$. If $a_1=-m$, then $M_{(i,b_1)}\neq 0$ for all $i\leq -m$ as $a_1-t\neq 0$ for $t\geq 0$. Suppose that $a_1\neq -m$. Then $M_{(i,b_1)}\neq 0 $ for $i\leq a_1$ as $a_1-t\neq 0$ for $t\geq 0$. 
 
 Now consider the exact sequence 
$$0\rightarrow H_m(\underline{x},M)_{(a+m,b)}\rightarrow M_{(a,b)}\rightarrow M^m_{(a+1,b)}.$$
 By Theorem \ref{Koszul vanish for underline x and y}, $H_m(\underline{x},M)_{(a+m,b)}=0$ whenever $a+m\neq 0$. Hence $M_{(a,b)}\neq 0$  implies  $M_{(a+1,b)}\neq 0$ if $a\neq -m$.  This implies that $M_{(i,b_1)}\neq 0$ for $a_1< i\leq -m$ as $a_1+t\neq -m$ when $0\leq t\leq -a_1-m-1$. Therefore $M_{(i,b_1)}\neq 0$ for all $i\leq -m$.

 Now fix $r\leq -m$, then $M_{(r,b_1)}\neq 0$. Consider the following exact sequence 
$$0\rightarrow H_n(\underline{y},M)_{(a,b+n)}\rightarrow M_{(a,b)}\rightarrow M^n_{(a,b+1)}.$$
 By Theorem \ref{Koszul vanish for underline x and y}, $H_n(\underline{y},M)_{(a, b+n)}=0$ whenever $b+n\neq 0$. Hence $M_{(a,b)}\neq 0$ implies $M_{(a,b+1)}\neq 0$ if $b\neq -n$. Therefore if $b_1=0$, then $M_{(r,j)}\neq 0$ for $j\geq 0$ as $b_1+t\neq -n$ for all $t\geq 0$.  If $b_1\neq 0$, then $M_{(r,j)}\neq 0$ for $j\geq b_1$ as $b_1+t\neq -n$ for all $t\geq 0$.
 
 Now consider the following exact sequence $$M^n_{(a,b-1)}\rightarrow M_{(a,b)}\rightarrow H_0(\underline{y},M)_{(a,b)}\rightarrow 0.$$
 By Theorem \ref{Koszul vanish for underline x and y}, $H_0(\underline{y},M)_{(a,b)}=0$ whenever $b\neq 0$. Hence $M_{(a,b)}\neq 0$ implies $M_{(a,b-1)}\neq 0$ if $b\neq 0$. This implies that $M_{(r,j)}\neq 0$ for $0\leq j<b_1$ as $b_1-t\neq 0$ when $0\leq t\leq b_1-1$. Therefore $M_{(r,j)}\neq 0$ for all $j\geq 0$. Now varying $r$ from $-m$ to $-\infty$, we get the result.

(3) We only need to prove that $(c)\implies (a)$ since $(a)\implies (b)\implies (c)$. Let $M_{(a_1,b_1)}\neq 0$ for some $(a_1,b_1)\in \mathbb{Z}^2$ with $a_1\leq -m$ and $b_1\leq -n$. Now consider the following exact sequence $$M^m_{(a-1,b)}\rightarrow M_{(a,b)}\rightarrow H_0(\underline{x},M)_{(a,b)}\rightarrow 0.$$
 By Theorem \ref{Koszul vanish for underline x and y}, $H_0(\underline{x},M)_{(a,b)}=0$ whenever $a\neq 0$. Hence $M_{(a,b)}\neq 0$ implies  $M_{(a-1,b)}\neq 0$ if $a\neq 0$. If $a_1=-m$, then $M_{(i,b_1)}\neq 0$ for all $i\leq -m$ as $a_1-t\neq 0$ for $t\geq 0$. Suppose that $a_1\neq -m$. Then $M_{(i,b_1)}\neq 0 $ for $i\leq a_1$ as $a_1-t\neq 0$ for $t\geq 0$. 
 
 Now consider the exact sequence 
$$0\rightarrow H_m(\underline{x},M)_{(a+m,b)}\rightarrow M_{(a,b)}\rightarrow M^m_{(a+1,b)}.$$
 By Theorem \ref{Koszul vanish for underline x and y}, $H_m(\underline{x},M)_{(a+m,b)}=0$ whenever $a+m\neq 0$. Hence $M_{(a,b)}\neq 0$  implies  $M_{(a+1,b)}\neq 0$ if $a\neq -m$.  This implies that $M_{(i,b_1)}\neq 0$ for $a_1< i\leq -m$ as $a_1+t\neq -m$ when $0\leq t\leq -a_1-m-1$. Therefore $M_{(i,b_1)}\neq 0$ for all $i\leq -m$.

Fix $r\leq -m$. Now consider the following exact sequence $$M^n_{(a,b-1)}\rightarrow M_{(a,b)}\rightarrow H_0(\underline{y},M)_{(a,b)}\rightarrow 0.$$
 By Theorem \ref{Koszul vanish for underline x and y}, $H_0(\underline{y},M)_{(a,b)}=0$ whenever $b\neq 0$. Hence $M_{(a,b)}\neq 0$ implies $M_{(a,b-1)}\neq 0$ if $b\neq 0$. If $b_1=-n$, then $M_{(r,j)}\neq 0$ for all $j\leq -n$ as $b_1-t\neq 0$ for $t\geq 0$. Suppose that $b_1\neq -n$. Then $M_{(r,j)}\neq 0 $ for $j\leq b_1$ as $b_1-t\neq 0$ for $t\geq 0$.
 
 Now consider the exact sequence 
$$0\rightarrow H_n(\underline{y},M)_{(a,b+n)}\rightarrow M_{(a,b)}\rightarrow M^n_{(a,b+1)}.$$
 By Theorem \ref{Koszul vanish for underline x and y}, $H_n(\underline{y},M)_{(a,b+n)}=0$ whenever $b+n\neq 0$. Hence $M_{(a,b)}\neq 0$ implies $M_{(a,b+1)}\neq 0$ if $b\neq -n$.  This implies that $M_{(r,j)}\neq 0$ for $b_1< j\leq -n$ as $b_1+t\neq -n$ when $0\leq t\leq -b_1-n-1$. Therefore $M_{(r,j)}\neq 0$ for all $j\leq -n$. Now varying $r$ from $-m$ to $-\infty$, we get the result.

(4) Proof of (4) is similar to (2).
\end{proof}
Now, we examine the case where a component, not included in the previous theorem, i.e., in \ref{Rigidity for K-1}, is non-zero. This is shown next.
\begin{theorem}[with hypothesis as in Theorem \ref{Rigidity for K-1}] \label{Rigidity for K-2}We have,
		\begin{enumerate}[\rm(1)]
		\item The following statements are equivalent:
		\begin{enumerate}[\rm(a)]
		\item $M_{(a,b)} \neq 0$ for all $(a,b)\in \mathbb{Z}^2$.
		\item There exists some $(-m,-n)<(a_0,b_0)<(0,0)$ such that $M_{(a_0,b_0)} \neq 0$.
	\end{enumerate}
\item The following statements are equivalent:
\begin{enumerate}[\rm (a)]
	\item $M_{(a,b)} \neq 0$ for all $(a,b)\in \mathbb{Z}^2$ with $b \geq 0$.
	\item There exists some $(a_0,b_0)\in \mathbb{Z}^2$ with $-m<a_0<0$ and $b_0 \geq 0$ such that $M_{(a_0,b_0)} \neq 0$.
\end{enumerate}
\item The following statements are equivalent:
\begin{enumerate}[\rm (a)]
	\item $M_{(a,b)} \neq 0$ for all $(a,b)\in \mathbb{Z}^2$ with $a \leq -m$.
	\item There exists some $(a_0,b_0)\in \mathbb{Z}^2$ with $a_0 \leq -m$ and $-n<b_0<0$ such that $M_{(a_0,b_0)} \neq 0$.
\end{enumerate}
\item The following statements are equivalent:
\begin{enumerate}[\rm (a)]
	\item $M_{(a,b)} \neq 0$ for all $(a,b)\in \mathbb{Z}^2$ with $b \leq -n$.
	\item There exists some $(a_0,b_0)\in \mathbb{Z}^2$ with $-m< a_0 <0$ and $b_0 \leq -n$ such that $M_{(a_0,b_0)} \neq 0$.
\end{enumerate}
\item The following statements are equivalent:
\begin{enumerate}[\rm (a)]
	\item $M_{(a,b)} \neq 0$ for all $(a,b)\in \mathbb{Z}^2$ with $a \geq 0$.
	\item There exists some $(a_0,b_0)\in \mathbb{Z}^2$ with $a_0 \geq 0$ and $-n<b_0<0$ such that $M_{(a_0,b_0)} \neq 0$.
\end{enumerate}
\end{enumerate}
\end{theorem}
\begin{proof}
The numbering of the figures corresponds to the parts of the theorem stated.

\begin{tikzpicture}[scale=0.15]
		\draw[->, ultra thick] (-7.5,0)--(8.5,0) node[right]{$x$};
	\draw[->, ultra thick] (0,-7.5)--(0,8.5) node[above]{$y$};
		\draw[dashdotted, red] (-7.5,-3)--(8,-3)node[below]{$y=-n$};
		\draw[dashdotted, red] (-2,-7.5)--(-2,8) node[left]{$x=-m$};
		\node at (-1,-1.5){$\bullet$};
		\draw[fill=brown,fill opacity=0.35,draw=none] (-7.5,8)--(8,8)--(8,-7.5)--(-7.5,-7.5);
		\node at (0,-11) {\textit Figure $1$};
		\end{tikzpicture}	
        \hspace{.3cm}
\begin{tikzpicture}[scale=0.15]
	\draw[->, ultra thick] (-7.5,0)--(8.5,0) node[right]{$x$};
	\draw[->, ultra thick] (0,-7.5)--(0,8.5) node[above]{$y$};
	\coordinate (1) at (-1,3);
	\foreach \x in {1}{
		\node[mynode] at (\x) {};}
	\draw[dashdotted, red] (8,-3)--(-7.5,-3)node[left]{$y=-n$};
	\draw[dashdotted, red] (-2,-7.5)--(-2,8) node[left]{$x=-m$};
	\draw[fill=brown,fill opacity=0.35,draw=none] (-8,8)--(8,8)--(8,0)--(-8,0);
	\node at (0,-11) {\textit{Figure $2$}};
	\end{tikzpicture}
    \hspace{0.3cm}
    \begin{tikzpicture}[scale=0.15]
	\draw[->, ultra thick] (-7.5,0)--(8.5,0) node[right]{$x$};
	\draw[->, ultra thick] (0,-7.5)--(0,8.5) node[above]{$y$};
	\coordinate (1) at (-5,-1.5);
	\foreach \x in {1}{
		\node[mynode] at (\x) {};}
	\draw[dashdotted, red] (8,-3)--(-7.5,-3)node[left]{$y=-n$};
	\draw[dashdotted, red] (-2,-7.5)--(-2,8) node[left]{$x=-m$};
	\draw[fill=brown,fill opacity=0.35,draw=none] (-8,8)--(-2,8)--(-2,-8)--(-8,-8);
	\node at (0,-11) {\textit{Figure $3$}};
	\end{tikzpicture}
    \begin{tikzpicture}[scale=0.15]
	\draw[->, ultra thick] (-7.5,0)--(8.5,0) node[right]{$x$};
	\draw[->, ultra thick] (0,-7.5)--(0,8.5) node[above]{$y$};
	\coordinate (1) at (-1,-5);
	\foreach \x in {1}{
		\node[mynode] at (\x) {};}
	\draw[dashdotted, red] (8,-3)--(-7.5,-3)node[left]{$y=-n$};
	\draw[dashdotted, red] (-2,-7.5)--(-2,8) node[left]{$x=-m$};
	\draw[fill=brown,fill opacity=0.35,draw=none] (-8,-3)--(8,-3)--(8,-8)--(-8,-8);
	\node at (0,-11) {\textit{Figure $4$}};
	\end{tikzpicture}
    \hspace{.5cm}
	\begin{tikzpicture}[scale=0.15]
	\draw[->, ultra thick] (-7.5,0)--(8.5,0) node[right]{$x$};
	\draw[->, ultra thick] (0,-7.5)--(0,8.5) node[above]{$y$};
	\coordinate (1) at (5,-1.5);
	\foreach \x in {1}{
		\node[mynode] at (\x) {};}
	\draw[dashdotted, red] (8,-3)--(-7.5,-3)node[left]{$y=-n$};
	\draw[dashdotted, red] (-2,-7.5)--(-2,8) node[left]{$x=-m$};
	\draw[fill=brown,fill opacity=0.35,draw=none] (0,8)--(8,8)--(8,-7.5)--(0,-7.5);
	\node at (0,-11) {\textit{Figure $5$}};
	\end{tikzpicture}
    
(1) We only need to prove $(b) \implies (a)$. Consider the following exact sequence 
$$0\rightarrow H_m(\underline{x},M)_{(a+m,b)}\rightarrow M_{(a,b)}\rightarrow M^m_{(a+1,b)}.$$
 By Theorem \ref{Koszul vanish for underline x and y}, $H_m(\underline{x},M)_{(a+m,b)}=0$ whenever $a+m\neq 0$. Hence $M_{(a,b)}\neq 0$  implies  $M_{(a+1,b)}\neq 0$ if $a\neq -m$. If $a_0+t=-m$ for some $t\geq 0$, then  $ a_0=-m-t\leq -m$, a contradiction. Since $M_{(a_0,b_0)}\neq 0$, this implies $M_{(i,b_0)}\neq 0$ for $i\geq a_0$.
 
 Now consider the following exact sequence $$M^m_{(a-1,b)}\rightarrow M_{(a,b)}\rightarrow H_0(\underline{x},M)_{(a,b)}\rightarrow 0.$$
 By Theorem \ref{Koszul vanish for underline x and y}, $H_0(\underline{x},M)_{(a,b)}=0$ whenever $a\neq 0$. Hence $M_{(a,b)}\neq 0$ implies  $M_{(a-1,b)}\neq 0$ if $a\neq 0$. If $a_0-t=0$ for some $t\geq 0$, then $a_0=t\geq 0$, a contradiction. Since $M_{(a_0,b_0)}\neq 0$, we get that $M_{(i,b_0)}\neq 0$ for all $i< a_0$. Hence $M_{(i,b_0)}\neq 0$ for all $i\in \mathbb{Z}$.

 Now fix $r\in \mathbb{Z}$.  Consider the following exact sequence 
$$0\rightarrow H_n(\underline{y},M)_{(a,b+n)}\rightarrow M_{(a,b)}\rightarrow M^n_{(a,b+1)}.$$
 By Theorem \ref{Koszul vanish for underline x and y}, $H_n(\underline{y},M)_{(a,b+n)}=0$ whenever $b+n\neq 0$. Hence $M_{(a,b)}\neq 0$ implies $M_{(a,b+1)}\neq 0$ if $b\neq -n$. If $b_0+t=-n$ for some $t\geq 0$, then  $ b_0=-n-t\leq -n$, a contradiction. Since $M_{(r,b_0)}\neq 0$, this implies $M_{(r,j)}\neq 0$ for $j\geq b_0$.

 Now consider the following exact sequence $$M^n_{(a,b-1)}\rightarrow M_{(a,b)}\rightarrow H_0(\underline{y},M)_{(a,b)}\rightarrow 0.$$
 By Theorem \ref{Koszul vanish for underline x and y}, $H_0(\underline{y},M)_{(a,b)}=0$ whenever $b\neq 0$. Hence $M_{(a,b)}\neq 0$ implies $M_{(a,b-1)}\neq 0$ if $b\neq 0$. If $b_0-t=0$ for some $t\geq 0$, then $b_0=t\geq 0$, a contradiction. Since $M_{(r,b_0)}\neq 0$, we get that $M_{(r,j)}\neq 0$ for all $j\leq b_0$. Hence $M_{(r,j)}\neq 0$ for all $j\in \mathbb{Z}$. Now varying $r$ from $-\infty$ to $+\infty$, we get our result.

(2)  We only need to prove $(b) \implies (a)$. Consider the following exact sequence 
$$0\rightarrow H_m(\underline{x},M)_{(a+m,b)}\rightarrow M_{(a,b)}\rightarrow M^m_{(a+1,b)}.$$
 By Theorem \ref{Koszul vanish for underline x and y}, $H_m(\underline{x},M)_{(a+m,b)}=0$ whenever $a+m\neq 0$. Hence $M_{(a,b)}\neq 0$  implies  $M_{(a+1,b)}\neq 0$ if $a\neq -m$. If $a_0+t=-m$ for some $t\geq 0$, then  $ a_0=-m-t\leq -m$, a contradiction. Since $M_{(a_0,b_0)}\neq 0$, this implies $M_{(i,b_0)}\neq 0$ for $i\geq a_0$.
 
 Now consider the following exact sequence $$M^m_{(a-1,b)}\rightarrow M_{(a,b)}\rightarrow H_0(\underline{x},M)_{(a,b)}\rightarrow 0.$$
 By Theorem \ref{Koszul vanish for underline x and y}, $H_0(\underline{x},M)_{(a,b)}=0$ whenever $a\neq 0$. Hence $M_{(a,b)}\neq 0$ implies  $M_{(a-1,b)}\neq 0$ if $a\neq 0$. If $a_0-t=0$ for some $t\geq 0$, then $a_0=t\geq 0$, a contradiction. Since $M_{(a_0,b_0)}\neq 0$, we get that $M_{(i,b_0)}\neq 0$ for all $i< a_0$. Hence $M_{(i,b_0)}\neq 0$ for all $i\in \mathbb{Z}$.

  Now fix $r\in \mathbb{Z}$, then $M_{(r,b_0)}\neq 0$. Consider the following exact sequence 
$$0\rightarrow H_n(\underline{y},M)_{(a,b+n)}\rightarrow M_{(a,b)}\rightarrow M^n_{(a,b+1)}.$$
 By Theorem \ref{Koszul vanish for underline x and y}, $H_n(\underline{y},M)_{(a, b+n)}=0$ whenever $b+n\neq 0$. Hence $M_{(a,b)}\neq 0$ implies $M_{(a,b+1)}\neq 0$ if $b\neq -n$. Therefore if $b_0=0$, then $M_{(r,j)}\neq 0$ for $j\geq 0$ as $b_0+t\neq -n$ for all $t\geq 0$. 

 If $b_0\neq 0$, then $M_{(r,j)}\neq 0$ for $i\geq b_0$ as $b_0+t\neq -n$ for all $t\geq 0$. Now consider the following exact sequence $$M^n_{(a,b-1)}\rightarrow M_{(a,b)}\rightarrow H_0(\underline{y},M)_{(a,b)}\rightarrow 0.$$
 By Theorem \ref{Koszul vanish for underline x and y}, $H_0(\underline{y},M)_{(a,b)}=0$ whenever $b\neq 0$. Hence $M_{(a,b)}\neq 0$ implies $M_{(a,b-1)}\neq 0$ if $b\neq 0$. This implies that $M_{(r,j)}\neq 0$ for $0\leq j<b_0$ as $b_0-t\neq 0$ when $0\leq t\leq b_0-1$. Therefore $M_{(r,j)}\neq 0$ for all $j\geq 0$. Now varying $r$ from $-\infty$ to $+\infty$, we get the result.

(3) We only need to prove $(b) \implies (a)$. Now consider the following exact sequence $$M^m_{(a-1,b)}\rightarrow M_{(a,b)}\rightarrow H_0(\underline{x},M)_{(a,b)}\rightarrow 0.$$
 By Theorem \ref{Koszul vanish for underline x and y}, $H_0(\underline{x},M)_{(a,b)}=0$ whenever $a\neq 0$. Hence $M_{(a,b)}\neq 0$ implies  $M_{(a-1,b)}\neq 0$ if $a\neq 0$. If $a_0=-m$, then $M_{(i,b_0)}\neq 0$ for all $i\leq -m$ as $a_0-t\neq 0$ for $t\geq 0$. 
 
 Suppose that $a_0\neq -m$. Then $M_{(i,b_0)}\neq 0 $ for $i\leq a_0$ as $a_0-t\neq 0$ for $t\geq 0$. Now consider the exact sequence 
$$0\rightarrow H_m(\underline{x},M)_{(a+m,b)}\rightarrow M_{(a,b)}\rightarrow M^m_{(a+1,b)}.$$
 By Theorem \ref{Koszul vanish for underline x and y}, $H_m(\underline{x},M)_{(a+m,b)}=0$ whenever $a+m\neq 0$. Hence $M_{(a,b)}\neq 0$  implies  $M_{(a+1,b)}\neq 0$ if $a\neq -m$.  This implies that $M_{(i,b_0)}\neq 0$ for $a_0< i\leq -m$ as $a_0+t\neq -m$ when $0\leq t\leq -a_0-m-1$. Therefore $M_{(i,b_0)}\neq 0$ for all $i\leq -m$.
 
 Now fix $r\leq -m$. Consider the following exact sequence 
$$0\rightarrow H_n(\underline{y},M)_{(a,b+n)}\rightarrow M_{(a,b)}\rightarrow M^n_{(a,b+1)}.$$
 By Theorem \ref{Koszul vanish for underline x and y}, $H_n(\underline{y},M)_{(a,b+n)}=0$ whenever $b+n\neq 0$. Hence $M_{(a,b)}\neq 0$ implies $M_{(a,b+1)}\neq 0$ if $b\neq -n$. If $b_0+t=-n$ for some $t\geq 0$, then  $ b_0=-n-t\leq -n$, a contradiction. Since $M_{(r,b_0)}\neq 0$, this implies $M_{(r,j)}\neq 0$ for $j\geq b_0$.

 Now consider the following exact sequence $$M^n_{(a,b-1)}\rightarrow M_{(a,b)}\rightarrow H_0(\underline{y},M)_{(a,b)}\rightarrow 0.$$
 By Theorem \ref{Koszul vanish for underline x and y}, $H_0(\underline{y},M)_{(a,b)}=0$ whenever $b\neq 0$. Hence $M_{(a,b)}\neq 0$ implies $M_{(a,b-1)}\neq 0$ if $b\neq 0$. If $b_0-t=0$ for some $t\geq 0$, then $b_0=t\geq 0$, a contradiction. Since $M_{(r,b_0)}\neq 0$, we get that $M_{(r,j)}\neq 0$ for all $j\leq b_0$. Hence $M_{(r,j)}\neq 0$ for all $j\in \mathbb{Z}$. Now varying $r$ from $-m$ to $-\infty$, we get our result.

(4) Proof of (4) is similar to (3).

(5)  Proof of (5) is similar to (2).
\end{proof}
Now we present examples corresponding to the regions described in Theorem \ref{Rigidity for K-1} and Theorem \ref{Rigidity for K-2}. These examples are taken from \cite{TP-big}. Although the authors in \cite{TP-big}  work in characteristic zero, the same examples hold in positive characteristic as well.
\begin{example}
Let $R=K[x_1,\ldots,x_m,y_1,\ldots,y_n]$ be standard bigraded over an infinite field $K$ of characteristic $p>0$. Let $\m=(x_1,\ldots,x_m,y_1,\ldots,y_n)$, $I_1=(x_1,\ldots,x_m)$, $I_2=(y_1,\ldots,y_n)$, $I_3=(x_1,\ldots,x_m,y_1)$, $I_4=(x_m,y_1,\ldots,y_n)$. Then the followings can be easily verified.
\begin{enumerate}
    \item $H^1_{(x_1)}(R)\cong x_1^{-1}K[x_1^{-1},x_2,\ldots,x_m,y_1,\ldots,y_n]$,
    \item $H^1_{(y_1)}(R)\cong y_1^{-1}K[x_1,\ldots,x_m,y_1^{-1},y_2,\ldots,y_n]$,
    \item $H^2_{(x_1,y_1)}(R)=x_1^{-1}y_1^{-1}K[x_1^{-1},x_2,\ldots,x_m,y_1^{-1},y_2,\ldots,y_n]$,
    \item $H^m_{I_1}(R)=x_1^{-1}\cdots x_m^{-1}K[x_1^{-1},\ldots,x_m^{-1},y_1,\ldots,y_n]$,
    \item $H^n_{I_2}(R)=y_1^{-1}\cdots y_n^{-1}K[x_1,\ldots,x_m,y_1^{-1},\ldots,y_n^{-1}]$,
    \item $H^{m+1}_{I_3}(R)=x_1^{-1}\cdots x_m^{-1}y_1^{-1}K[x_1^{-1},\ldots,x_m^{-1},y_1^{-1},y_2,\ldots,y_n]$,
    \item  $H^{n+1}_{I_4}(R)=x_m^{-1}y_1^{-1}\cdots y_n^{-1}K[x_1,\ldots,x_{m-1},x_m^{-1},y_1^{-1},\ldots,y_n^{-1}]$,
    \item $H^{m+n}_{\m}=x_1^{-1}\cdots x_m^{-1}y_1^{-1}\cdots y_n^{-1}K[x_1^{-1},\ldots,x_m^{-1},y_1^{-1},\ldots,y_n^{-1}]$,
    \item $H^0_{(0)}(R)=R$.
\end{enumerate}
Examples 9, 4, 8 and 5 correspond to Figures 1, 2, 3 and 4 respectively in Theorem \ref{Rigidity for K-1}, while examples 3, 1, 6, 7 and 2 correspond to Figures 1, 2, 3, 4 and 5 respectively in Theorem \ref{Rigidity for K-2}.
\end{example}
\section{Rigidity properties when \texorpdfstring{$R$}{R} is polynomial ring over \texorpdfstring{$A$}{A}}
Our setup for this section is as follows:
\s \textit{Setup:}
  Let $A$ be a regular ring containing a field of characteristic $p>0$. Let $R=A[x_1,\ldots,x_m,y_1,\ldots,y_n]$ be standard bigraded over $A$, i.e., $\bideg(A)=(0,0)$, $\bideg(x_i)=(1,0)$ and  $\bideg(y_j)=(0,1)$ for all $i$ and $j$ and  $M=\bigoplus_{i,j} M_{(i,j)}$, a bigraded $F_R$-finite, $F_R$-module.
  
In this section, we study the rigidity properties of bigraded components $M$. Now, let us make the following remark, which we will frequently use to prove several results.
\begin{remark}
\label{stan}
    Suppose $M_{(a,b)}\neq 0$ for some $(a,b)\in \mathbb{Z}^2$. Let $P$ be a minimal prime ideal of $M_{(a,b)}$ and let $B=\widehat{A_P}$. Set $S=B[x_1,\ldots,x_m,y_1,\ldots,y_n]$. By \ref{pi take f finite to f finite}, $N=S\otimes_RM=B\otimes_A M$ is $F_S$-finite, $F_S$-module. By Cohen structure theorem $B=K[[t_1,\ldots,t_g]]$ where $K=\kappa(P)$ is the residue field of $B$. We note that $N_{(a,b)}={(M_{(a,b)})_P}\neq 0$. As $P$ is minimal prime ideal of $M_{(a,b)}$ so $N_{(a,b)}$ is supported only at the maximal ideal of $B$. Therefore by Theorem \ref{ordinal}, $N_{(a,b)}=E_B(K)^{\alpha}$ for some $\alpha$. 

    If $K$ is finite we consider an infinite field $K'$ containing $K$ and consider the flat extension $B\rightarrow C=K'[[t_1,\ldots,t_g]]$. Set $T=C[x_1,\ldots,x_m,y_1,\ldots,y_n]$ a flat extension of $S$. By \ref{pi take f finite to f finite}, $L=T\otimes_S N=C\otimes_B N$ is $F_T$-finite, $F_T$-module. Therefore $L_{(a,b)}\neq 0$ and supported only at the maximal ideal of $C$. We note that $L_{(a,b)}=E_C(K')^{\alpha}$.

    Let $V=H_g(t_1,\ldots,t_g;L)$. Then we have that $V$ is  $F_D$-finite, $F_D$-module where $D=K'[x_1,\ldots,x_m,y_1,\ldots,y_n]$. Also $V\subseteq L$ and $L=M\otimes_A C$.
\end{remark}
Our first result in the sequel is as follows:
\begin{theorem}\label{Vanishing of M-A}
   Let $A$ be a regular ring containing a field of characteristic $p>0$. Let $R=A[x_1,\ldots,x_m,y_1,\ldots,y_n]$ be standard bigraded over $A$, i.e., $\bideg(A)=(0,0)$, $\bideg(x_i)=(1,0)$ and  $\bideg(y_j)=(0,1)$ for all $i$ and $j$ and  $M=\bigoplus_{i,j} M_{(i,j)}$, a bigraded $F_R$-finite, $F_R$-module. If $M_{(m,n)}=0$ for all $|m|\gg 0 $ and $|n|\gg 0$, then $M=0$.
\end{theorem}
\begin{proof}
Suppose if possible $M_{(a,b)}\neq 0$ for some $(a,b)\in \mathbb{Z}^2$. Let $L$ be as in the Remark \ref{stan}. Now $V_{(m,n)}=0$ for $|m|\gg 0$ and $|n|\gg 0$ since $V\subseteq L$. So by \ref{Vanishing of M}, $V=0$ since $V$ is $F_D$-finite, $F_D$-module where $D=K'[x_1,\ldots,x_m,y_1,\ldots,y_n]$. But this is a  contradiction to the fact that $V_{(a,b)}\neq 0$.
\end{proof}
The next two results are generalization of Theorem \ref{Rigidity for K-1} and \ref{Rigidity for K-2}, which we previously considered in the case where the ring $A$ was a field. 
\begin{theorem}\label{Rigidity for A-1}
 Let $A$ be a regular ring containing a field of characteristic $p>0$. Let $R=A[x_1,\ldots,x_m,y_1,\ldots,y_n]$ be standard bigraded over $A$, i.e., $\bideg(A)=(0,0)$, $\bideg(x_i)=(1,0)$ and  $\bideg(y_j)=(0,1)$ for all $i$ and $j$ and  $M=\bigoplus_{i,j} M_{(i,j)}$, a bigraded $F_R$-finite, $F_R$-module. Then,
  \begin{enumerate}[\rm (1)]
     \item The following statements are equivalent:
     \begin{enumerate}[\rm (a)]
     \item $M_{(a,b)}\neq 0$ for all $(a,b)\geq (0,0)$.
         \item There exists $(a_0,b_0)\geq (0,0)$ such that $M_{(a,b)}\neq 0$ for all $(a,b)\geq (a_0,b_0)$.
         \item $M_{(a_1,b_1)}\neq 0$ for some $(a_1,b_1)\geq (0,0)$.
     \end{enumerate}
     \item The following statements are equivalent:
     \begin{enumerate}[\rm (a)]
         \item $M_{(a,b)}\neq 0$ for all $(a,b)\in\mathbb{Z}^2$ with $a\leq -m$ and $b\geq 0$.
         \item There exists $(a_0,b_0)\in \mathbb{Z}^2$ with $a_0\leq -m$ and $b_0\geq 0$ such that $M_{(a,b)}\neq 0$ for all $(a,b)\in \mathbb{Z}^2$ with $a\leq a_0$ and $b\geq b_0$.
         \item $M_{(a_1,b_1)}\neq 0$ for some $(a_1,b_1)\in \mathbb{Z}^2$ with $a_1\leq -m$ and $b_1\geq 0$.
     \end{enumerate}
     \item The following statements are equivalent:
     \begin{enumerate}[\rm (a)]
         \item $M_{(a,b)}\neq 0$ for all $(a,b)\in\mathbb{Z}^2$ with $(a,b)\leq (-m,-n)$.
         \item There exists $(a_0,b_0)\in \mathbb{Z}^2$ with $(a_0,b_0)\leq (-m,-n)$  such that $M_{(a,b)}\neq 0$ for all $(a,b)\leq (a_0,b_0)$.
         \item $M_{(a_1,b_1)}\neq 0$ for some $(a_1,b_1)\in \mathbb{Z}^2$ with $(a_1,b_1)\leq (-m,-n)$.
     \end{enumerate}
\item The following statements are equivalent:
     \begin{enumerate}[\rm (a)]
         \item $M_{(a,b)}\neq 0$ for all $(a,b)\in\mathbb{Z}^2$ with $a\geq 0$ and $b\leq -n$.
         \item There exists $(a_0,b_0)\in \mathbb{Z}^2$ with $a_0\geq 0$ and $b_0\leq -n$ such that $M_{(a,b)}\neq 0$ for all $(a,b)\in \mathbb{Z}^2$ with $a\geq a_0$ and $b\leq b_0$.
         \item $M_{(a_1,b_1)}\neq 0$ for some $(a_1,b_1)\in \mathbb{Z}^2$ with $a_1\geq 0$ and $b_1\leq -n$.
     \end{enumerate}
 \end{enumerate}
\end{theorem}
\begin{proof}
   (1) We only need to prove that $(c)\implies (a)$ since $(a)\implies (b)\implies (c)$. Suppose $M_{(a_1,b_1)}\neq 0$ for some $(a_1,b_1)\geq (0,0)$. Let $P$ be a minimal prime of $M_{(a_1,b_1)}$. Again by Remark \ref{stan}, $L_{(a_1,b_1)}\neq 0$. Now $V$ is $F_D$-finite, $F_D$-module where $D=K'[x_1,\ldots,x_m,y_1,\ldots,y_n]$ with $V_{(a_1,b_1)}\neq 0$. By Theorem \ref{Rigidity for K-1}, $V_{(a,b)}\neq 0$ for all $(a,b)\geq(0,0)$. Since $V_{(a,b)}\subseteq L_{(a,b)}$ so $L_{(a,b)}\neq 0$ for all $(a,b)\geq (0,0)$. Therefore $M_{(a,b)}\neq 0$ for all $(a,b)\geq (0,0)$ as $L_{(a,b)}=M_{(a,b)}\otimes_A C$. 
   
   The proofs of parts 2, 3, and 4 are identical to that of part 1; we simply need to apply Remark \ref{stan} and use Theorem \ref{Rigidity for K-1}.
\end{proof}
Similarly using Remark \ref{stan} and Theorem \ref{Rigidity for K-2} as in Theorem \ref{Rigidity for A-1}, we obtain the following.
\begin{theorem}\label{Rigidity for A-2}
Let $A$ be a regular ring containing a field of characteristic $p>0$. Let $R=A[x_1,\ldots,x_m,y_1,\ldots,y_n]$ be standard bigraded over $A$, i.e., $\bideg(A)=(0,0)$, $\bideg(x_i)=(1,0)$ and  $\bideg(y_j)=(0,1)$ for all $i$ and $j$ and  $M=\bigoplus_{i,j} M_{(i,j)}$, a bigraded $F_R$-finite, $F_R$-module. Then,
\begin{enumerate}[\rm(1)]
		\item The following statements are equivalent:
		\begin{enumerate}[\rm(a)]
		\item $M_{(a,b)} \neq 0$ for all $(a,b)\in \mathbb{Z}^2$.
		\item There exists some $(-m,-n)<(a_0,b_0)<(0,0)$ such that $M_{(a_0,b_0)} \neq 0$.
	\end{enumerate}
\item The following statements are equivalent:
\begin{enumerate}[\rm (a)]
	\item $M_{(a,b)} \neq 0$ for all $(a,b)\in \mathbb{Z}^2$ with $b \geq 0$.
	\item There exists some $(a_0,b_0)\in \mathbb{Z}^2$ with $-m<a_0<0$ and $b_0 \geq 0$ such that $M_{(a_0,b_0)} \neq 0$.
\end{enumerate}
\item The following statements are equivalent:
\begin{enumerate}[\rm (a)]
	\item $M_{(a,b)} \neq 0$ for all $(a,b)\in \mathbb{Z}^2$ with $a \leq -m$.
	\item There exists some $(a_0,b_0)\in \mathbb{Z}^2$ with $a_0 \leq -m$ and $-n<b_0<0$ such that $M_{(a_0,b_0)} \neq 0$.
\end{enumerate}
\item The following statements are equivalent:
\begin{enumerate}[\rm (a)]
	\item $M_{(a,b)} \neq 0$ for all $(a,b)\in \mathbb{Z}^2$ with $b \leq -n$.
	\item There exists some $(a_0,b_0)\in \mathbb{Z}^2$ with $-m< a_0 <0$ and $b_0 \leq -n$ such that $M_{(a_0,b_0)} \neq 0$.
\end{enumerate}
\item The following statements are equivalent:
\begin{enumerate}[\rm (a)]
	\item $M_{(a,b)} \neq 0$ for all $(a,b)\in \mathbb{Z}^2$ with $a \geq 0$.
	\item There exists some $(a_0,b_0)\in \mathbb{Z}^2$ with $a_0 \geq 0$ and $-n<b_0<0$ such that $M_{(a_0,b_0)} \neq 0$.
\end{enumerate}
\end{enumerate}
\end{theorem}
\section{Bass numbers}
Our setup for this section is as follows:

\s \textit{Setup:} \label{bass setup} Let $A$ be a regular ring containing a field of characteristic $p>0$. Let $R=A[x_1,\ldots,x_m,y_1,\ldots,y_n]$ be standard bigraded over $A$ and  $M=\bigoplus_{i,j} M_{(i,j)}$ , a bigraded $F_R$-finite, $F_R$-module.

We have the following interesting fact when $m=1$ and $n=1$.
\begin{theorem} \label{length for (1,1)}
    Let $K$ be an infinite field of characteristic $p>0$. Assume that $M=\bigoplus_{i,j} M_{(i,j)}$ is bigraded $F_R$-finite, $F_R$-module where $R=K[x_1,y_1]$ is standard bigraded. Then $\dim_K(M_{(i,j)})<\infty$ for all $(i,j)\in\mathbb{Z}^2$.
\end{theorem}
\begin{proof}
    We may assume that $K$ is algebraically closed. Let $\D$ be the ring of $K$-linear differential operators on $R$, i.e., $\D=R\left\langle\partial_{[i]},\delta_{[j]} \mid 1 \leq i,j \right\rangle$ where $\partial_{[i]}=\frac{1}{i!}\frac{\partial^i}{\partial x_1^i}$ and $\delta_{[j]}=\frac{1}{j!}\frac{\delta^j}{\delta y_1^j}$. Clearly $\D$ is bigraded with $\bideg(x_1)=(1,0)$ and $\bideg(y_1)=(0,1)$ and $\bideg(\partial_{[i]})=(-i,0)$, $\bideg(\delta_{[j]})=(0,-j)$. We note that $\D_{(0,0)}=K\left\langle x_1^i\partial_{[i]}, y_1^j\delta_{[j]}\mid i,j\geq 1\right\rangle$. By \ref{d-mod}, $M$ is bigraded Eulerian $\D$-module. Since $M$ is $F$-finite, $F$-module so by \cite[5.7]{Lyu-Fmod}, $M$ has finite length as $\D$-module. Therefore it can easily be checked that $M_{(i,j)}$ is Noetherian as $D_{(0,0)}$-module for any $(i,j)\in\mathbb{Z}^2$. Indeed if $U$ is a $D_{(0,0)}$-sumodule of $M_{(i,j)}$ then $\D U\cap M_{(i,j)}=U$. If $$U_1\subseteq U_2\subseteq\ldots\subseteq U_r\subseteq U_{(r+1)}\subseteq\ldots$$ is an ascending chain of $D_{(0,0)}$-submodule of $M_{(i,j)}$, then we have ascending chain of $D$-submodules of $M$ 
$$\D U_1\subseteq \D U_2\subseteq\ldots\subseteq \D U_r\subseteq \D U_{r+1}\subseteq\ldots.$$ Since $M$ is Noetherian there exists $t$ such that $\D U_r=\D U_t$ for all $r\geq t$. Intersecting with $M_{(i,j)}$, we get $U_r=U_t$ for all $r\geq t$. Hence $M_{(i,j)}$ is Noetherian as $D_{(0,0)}$-module.
In particular, $M_{(i,j)}$ is finitely generated $D_{(0,0)}$-module. Assume that $m_1,\ldots,m_s$ be the finite generators of $M_{(i,j)}$. Let $\bideg(m_i)=(m_1^i,m_2^i)$ for all $1\leq i\leq s$. Since $M$ is Eulerian 
$$E^X_r m_i=\binom{m^i_1}{r}m_i\  \text{and}\   E^Y_r m_i=\binom{m^i_2}{r}m_i$$ for all $r\geq 1$. This implies $D_{(0,0)}M_{(i,j)}\subseteq K m_1+\ldots+Km_s$. Hence $\dim_K(M_{(i,j)})<\infty$ for all $(i,j)\in\mathbb{Z}^2$.
\end{proof}
Let us prove the following result which is crucial for proving the main result on Bass numbers of this section.
\begin{theorem}\label{length as k vs}
    Let $K$ be an infinite field of characteristic $p>0$. Assume that $M=\bigoplus_{i,j} M_{(i,j)}$ is bigraded $F_R$-finite, $F_R$-module where $R=K[x_1,\ldots, x_m,y_1,\ldots,y_n]$ is standard bigraded. Then,
    \begin{enumerate}[\rm (1)]
        \item  If  $l(M_{(a,b)})$ is finite for some $(a,b)\in \mathbb{Z}^2$ with $(a,b)\geq(0,0)$, then $ l(M_{(r,s)})$ is finite for all $(r,s)\geq (0,0)$. 
        \item If $l(M_{(a,b)})$ is finite for some $(a,b)\in \mathbb{Z}^2$ with $a<0$ and $b\geq 0$, then $ l(M_{(r,s)})$ is finite for all $(r,s)\in\mathbb{Z}^2$ with $r<0$ and $s\geq 0$.
        \item If $l(M_{(a,b)})$ is finite for some $(a,b)\in \mathbb{Z}^2$ with $(a,b)<(0,0)$, then $ l(M_{(r,s)})$ is finite for all $(r,s)<(0,0)$.
        \item  If $l(M_{(a,b)})$ is finite for some  $(a,b)\in \mathbb{Z}^2$ with $a\geq 0$ and $b<0$, then $ l(M_{(r,s)})$ is finite for all $(r,s)\in \mathbb{Z}^2$ with $r\geq 0$ and $s<0$.
        \end{enumerate}
\end{theorem}
\begin{proof}We prove only (1). The proof of other parts follows similarly.

 The result for $m=1$ and $n=1$ follows from \ref{length for (1,1)}.
     
    Now let $m\geq 2$ and $n=1$. Let us first prove that $l(M_{(i,b)})$ is finite for all $i\in \mathbb{Z}$. If $m=2$, then by Theorem \ref{F-finite}, $H_0(x_1,M)$ and $H_1(x_1,M)$ are $F$-finite, $F$-module over $K[x_2,y_1]$. Therefore by base case $l(H_0(x_1,M)_{(i,j)})$ is finite for all $(i,j)\in \mathbb{Z}^2$. Similarly $l(H_1(x_1,M)_{(i,j)})$ is finite for all $(i,j)\in\mathbb{Z}^2$. Consider the following exact sequence $$0\rightarrow H_1(x_1,M)_{(i,b)}\rightarrow M_{(i-1,b)}\rightarrow M_{(i,b)} \rightarrow H_0(x_1,M)_{(i,b)}\rightarrow 0.$$
    As $l(H_0(x_1,M)_{(i,b)})$ and $l(H_1(x_1,M)_{(i,b)})$ is finite  for all $(i,b)$ so $l(M_{(i,b)})$ is finite $\iff l(M_{(i-1,b)})$ is finite. Since $l(M_{(a,b)})$ is finite we conclude that $l(M_{(i,b)})$ is finite for all $i\in \mathbb{Z}$.  Now assume this result for $m-1$ with $m\geq 2$ and $n=1$ and prove it for $m$ and $n=1$. Then by Theorem \ref{F-finite}, $H_0(x_1,M)$ and $H_1(x_1,M)$ are $F$-finite, $F$-module over $K[x_2,\ldots,x_m,y_1]$. Since $l(H_0(x_1,M)_{(a,b)})\leq l(M_{(a,b)})< \infty$, so by induction hypothesis $l(H_0(x_1,M)_{(i,b)})$ is finite for all $(i,b)\in \mathbb{Z}^2$. Similarly $l(H_1(x_1,M)_{(i,b)})$ is finite for all $(i,b)\in \mathbb{Z}^2$. Now the remaining argument is same as $m=2$ case.

    Fix an $i_0$. Consider the following exact sequence $$0\rightarrow H_1(y_1,M)_{(i_0,j)}\rightarrow M_{(i_0,j-1)}\rightarrow M_{(i_0,j)} \rightarrow H_0(y_1,M)_{(i_0,j)}\rightarrow 0.$$
By Proposition \ref{Koszul Vanish for (m,1),(1,n)}, $H_1(y_1,M)_{(i_0,j)}=0$ and $H_0(y_1,M)_{(i_0,j)}=0$ for $j\neq 0$. Therefore $M_{(i_0,j-1)}\cong  M_{(i_0,j)}$ for $j\neq 0$.
Since $l(M_{(i_0,b)})<\infty$ we have that $l(M_{(i_0,j)})<\infty$ for all $j\geq 0$. Since $i_0$ is arbitrary, $l(M_{(r,s)})<\infty$ for all $r \in \mathbb{Z}$ and $s\geq 0$. In particular,  $l(M_{(r,s)})<\infty$ for all $(r,s)\geq (0,0)$.

   Let $m=1$ and $n\geq 2$. Similarly we can prove that $l(M_{(r,s)})<\infty$ for all $r\geq 0$ and  $s \in \mathbb{Z}$. In particular,  $l(M_{(r,s)})<\infty$ for all $(r,s)\geq (0,0)$.

    Now assume the result holds for $(m,n-1)$ and $(m-1,n)$. Let us denote by $\overline{R}=K[x_2,\ldots,x_m,y_1,\ldots,y_n]$  and  $\overline{S}=K[x_1,\ldots,x_n,y_2,\ldots,y_n]$. By \ref{F-finite},  $H_i{(x_1,M)}$ and  $H_i{(y_1,M)}$ are $F$-finite over $\overline{R}$ and $\overline{S}$ respectively for $i=0,1$. We note that $l(H_0(x_1,M)_{(a,b)})\leq l(M_{(a,b)})<\infty$. Therefore by induction $l(H_0(x_1,M)_{(r,s)})<\infty$ for all $(r,s)\geq (0,0)$. Also $l(H_1(x_1,M)_{(a+1,b)})\leq l(M_{(a,b)})<\infty$. Again by induction $l(H_1(x_1,M)_{(r,s)})<\infty$ for all $(r,s)\geq (0,0)$. Consider the following exact sequence $$0\rightarrow H_1(x_1,M)_{(i,b)}\rightarrow M_{(i-1,b)}\rightarrow M_{(i,b)} \rightarrow H_0(x_1,M)_{(i,b)}\rightarrow 0.$$
    As $l(H_0(x_1,M)_{(i,b)})$ and $l(H_1(x_1,M)_{(i,b)})$ is finite  for all $i\geq 0$ so $l(M_{(i,b)})$ is finite $\iff l(M_{(i-1,b)})$ is finite when $i\geq 0$. Since $l(M_{(a,b)})$ is finite we conclude that $l(M_{(i,b)})$ is finite for all $i\geq 0$.  
 Now for a fixed $i$ considering the following exact sequence 
 $$0\rightarrow H_1(y_1,M)_{(i,j)}\rightarrow M_{(i,j-1)}\rightarrow M_{(i,j)}\rightarrow H_0(y_1,M)_{(i,j)}\rightarrow 0$$
and similar argument proves our result.
\end{proof}
We need the following Lemma from \cite[1.4]{Lyu-Dmod}.
\begin{lemma}\label{Lyu result on bass number}
    Let $B$ be a Noetherian ring and let $N$ be a $B$-module. Let $P$ be a prime ideal in $B$. If $(H^j_P(N))_P$ is injective for all $j\geq 0$ then $\mu_j(P,N)=\mu_0(P,H^j_P(N))$ for $j\geq 0$.
\end{lemma}
We now show that the hypothesis of the last stated lemma is satisfied in our case:
\begin{proposition}\label{Injective satisfy} (with hypothesis as in \ref{bass setup}) Let $P$ be prime ideal of $A$. Set $E=M_{(a,b)}$ for some $(a,b)\in \mathbb{Z}^2$. Then $H^j_P(E)_P$ is injective for all $j\geq 0$.
\end{proposition}
\begin{proof}
The proof proceeds along parallel lines to  \cite[Proposition 9.3]{TP-koszul}.   
\end{proof}
We now prove the main result of this section regarding the Bass numbers of bigraded components.
\begin{theorem}\label{Bass th for 1st quad}
    With hypothesis as in \ref{bass setup}, the following statements are true;
    \begin{enumerate}[\rm (1)]
        \item  If  $\mu_j(P,M_{(a,b)})$ is finite for some $(a,b)\in \mathbb{Z}^2$ such that $(a,b)\geq(0,0)$, then $ \mu_j(P,M_{(r,s)})$ is finite for all $(r,s)\geq (0,0)$. 
        \item  If $\mu_j(P,M_{(a,b)})$ is finite for some $(a,b)\in \mathbb{Z}^2$ such that $a<0$ and $b\geq 0$, then $ \mu_j(P,M_{(r,s)})$ is finite for all $(r,s)\in\mathbb{Z}^2$ with $r<0$ and $s\geq 0$.
        \item If  $\mu_j(P,M_{(a,b)})$ is finite for some $(a,b)\in \mathbb{Z}^2$ such that $(a,b)<(0,0)$, then $ \mu_j(P,M_{(r,s)})$ is finite for all $(r,s)<(0,0)$.
        \item  If  $\mu_j(P,M_{(a,b)})$ is finite for some  $(a,b)\in \mathbb{Z}^2$ such that $a\geq 0$ and $b<0$, then $ \mu_j(P,M_{(r,s)})$ is finite for all $(r,s)\in \mathbb{Z}^2$ with $r\geq 0$ and $s<0$.
        \end{enumerate}
\end{theorem}
\begin{proof}
    By Lemma \ref{Injective satisfy}, and Proposition \ref{Lyu result on bass number}, $\mu_j(P,M_{(a,b)})=\mu_0(P,H^j_P(M_{(a,b)}))$ for all $(a,b)\in \mathbb{Z}^2$.  Applying the technique of \ref{stan} to the bigraded $F$-finite, $F$-module $H^j_{PR}(M)$, we get that $N_{(a,b)}=H^j_P(M_{(a,b)})_P$ for all $(a,b)\in \mathbb{Z}^2$. Therefore we get that $N_{(a,b)}=E_B(K)^{\alpha_{(a,b)}}$ and may assume that $K$ is infinite. Note that $\alpha_{(a,b)}=\mu_j(P,M_{(a,b)})$ and $V=H_g(t_1,\ldots,t_g;N)$ is $F_D$-finite, $F_D$-module where $D=K[x_1,\ldots,x_m,y_1,\ldots,y_n]$. As $V_{(a,b)}=H_g(t_1,\ldots,t_g;N)_{(a,b)}=K^{\alpha_{(a,b)}}$, we get that $\dim_KV_{(a,b)}=\alpha_{(a,b)}=\mu_j(P,M_{(a,b)})$. The result now follows from \ref{length as k vs}.
\end{proof}
\section{Associated primes}
The setup for this section is the same as in the previous two sections, as stated below:

\s \textit{Setup:} \label{Associted section setup}Let $A$ be a regular ring containing a field of characteristic $p>0$. Let $R=A[x_1,\ldots,x_m,y_1,\ldots,y_n]$ be standard bigraded over $A$, i.e., $\bideg(A)=(0,0)$, $\bideg(x_i)=(1,0)$ and  $\bideg(y_j)=(0,1)$ for all $i$ and $j$ and  $M=\bigoplus_{i,j} M_{(i,j)}$, a bigraded $F_R$-finite, $F_R$-module.

The next theorem generalizes Theorem 1.12 from \cite{TP-koszul} to the bigrade setup. More precisely, we establish the following result:
 \begin{theorem}\label{Associated prime result}
	 With hypothesis as in \ref{Associted section setup}, the following statements are true; 
	
	\begin{enumerate}[\rm (1)]
		\item
		$\bigcup_{(a,b) \in \mathbb{Z}^2} \Ass_A M_{(a,b)}   $ is a finite set.
		\item
		$\Ass_A M_{(a,b)} = \Ass_A M_{(0,0)}$ for all $(a,b) \geq (0,0)$.
		\item
		$\Ass_A M_{(a,b)} = \Ass_A M_{(-m,0)}$ for all $(a,b) \in \mathbb{Z}^2$ with $a \leq -m$ and $b \geq 0$.
		\item
		$\Ass_A M_{(a,b)} = \Ass_A M_{(-m,-n)}$ for all $(a,b)\leq (-m,-n)$.
		\item
		$\Ass_A M_{(a,b)} = \Ass_A M_{(0,-n)}$ for all $(a,b) \in \mathbb{Z}^2$ with $a \geq 0$ and $b \leq -n$.
	\end{enumerate}
\end{theorem}
\begin{center}
	\begin{tikzpicture}[scale=0.2]
	\draw[->, ultra thick] (-8.5,0)--(8.5,0) node[right]{$x$};
	\draw[->, ultra thick] (0,-8.5)--(0,8.5) node[above]{$y$};
	\draw[dotted, red] (-2,8.5)--(-2,-8.5);
	\draw[densely dotted, red] (-8.5,-3)--(8.5, -3);
	\draw[fill=brown,fill opacity=0.35,draw=none] (0,8.5)--(0,0)--(8.5,0)--(8.5,8.5)--(0,8.5);
	\draw[fill=brown,fill opacity=0.35,draw=none] (-2,8.5)--(-2,0)--(-8.5,0)--(-8.5,8.5)--(-2,8.5);
	\draw[fill=brown,fill opacity=0.35,draw=none] (-2,-8.5)--(-2,-3)--(-8.5,-3)--(-8.5,-8.5)--(-2,-8.5);
	\draw[fill=brown,fill opacity=0.35,draw=none] (0,-8.5)--(0,-3)--(8.5,-3)--(8.5,-8.5)--(0,-8.5);
	\node[blue,draw=none] at (3,-5) {\tiny $(0,-n)$};
	\node  at (0,-3){$\bullet$};
	\node[blue,draw=none] at (2.5,1.5) {\tiny $(0,0)$};
	\node  at (0,0){$\bullet$};
	\node[blue,draw=none] at (-5,1.5) {\tiny $(-m,0)$};
	\node  at (-2,0){$\bullet$};
	\node[blue,draw=none] at (-5.5,-5) {\tiny $(-m,-n)$};
	\node at (-2,-3){$\bullet$};
	\node  at (0,-11.5) {\textit{Figure $1$}};
	\end{tikzpicture}
\end{center}
\begin{proof}
From the proof of Theorem \ref{Bass th for 1st quad} assuming $j=0$, we get that $\dim_KV_{(a,b)}=\alpha_{(a,b)}=\mu_0(P,M_{(a,b)})$. Here $V$ is a $F$-finite, $F$-module over $K[x_1,\ldots,x_m,y_1,\ldots,y_n]$ and $K$ is infinite.
\begin{enumerate}[(1)]
\item From \cite[2.12]{Lyu-Fmod} $\Ass_R(M)$ is finite  and therefore the result follows from Proposition \ref{ass-contraction}.

For $(2), (3), (4), (5)$ let 
$$
\bigcup_{(a,b) \in \mathbb{Z}^2} \Ass_A M_{(a,b)}  = \{ P_1, \ldots, P_l \}.
$$

Let $P = P_i$ for some $i$. Then by Theorem \ref{Rigidity for K-1}  we get that

	\item $\mu_0(P, M_{(0,0)}) > 0$ if and only if $\dim_KV_{(0,0)}\neq 0$  if and only if $\dim_KV_{(a,b)}\neq 0$  for all $(a,b)\geq (0,0)$ if and only if $\mu_0(P, M_{(a,b)}) > 0$ for all $(a,b) \geq (0,0)$.
    \item $\mu_0(P, M_{(-m,0)}) > 0$ if and only if $\dim_KV_{(-m,0)}\neq 0$  if and only if $\dim_KV_{(a,b)}\neq 0$  for all $a \leq -m$ and $b \geq 0$ if and only if $\mu_0(P, M_{(a,b)}) > 0$ for all $a \leq -m$ and $b \geq 0$.  
    \item $\mu_0(P, M_{(-m,-n)}) > 0$ if and only if $\dim_KV_{(-m,-n)}\neq 0$  if and only if $\dim_KV_{(a,b)}$\ $\neq 0$  for all $(a,b) \leq (-m,-n)$ if and only if $\mu_0(P, M_{(a,b)}) > 0$ for all $(a,b) \leq (-m,-n)$.
    \item $\mu_0(P, M_{(0,-n)}) > 0$ if and only if $\dim_KV_{(0,-n)}\neq 0$  if and only if $\dim_KV_{(a,b)}\neq 0$  for all $a \geq 0$ and $b \leq -n$ if and only if $\mu_0(P, M_{(a,b)}) > 0$ for all $a \geq 0$ and $b \leq -n$. 
\end{enumerate}  The result follows.	
\end{proof}
\section{Injective dimension and support dimension}
Again we have same setup as before i,e.,
\s \textit{Setup:} \label{Inj dim setup}Let $A$ be a regular ring containing a field of characteristic $p>0$. Let $R=A[x_1,\ldots,x_m,y_1,\ldots,y_n]$ be standard bigraded over $A$, i.e., $\bideg(A)=(0,0)$, $\bideg(x_i)=(1,0)$ and  $\bideg(y_j)=(0,1)$ for all $i$ and $j$ and  $M=\bigoplus_{i,j} M_{(i,j)}$, a bigraded $F_R$-finite, $F_R$-module.

\begin{lemma}\label{injdim-dim}
	(with hypotheses as in \ref{Inj dim setup}). Then for all $(a,b) \in \mathbb{Z}^2$,
	\[
	\injdim M_{(a,b)} \leq \dim M_{(a,b)}.
	\]
\end{lemma}
\begin{proof}
	Let $P$ be a prime ideal in $A$. Then Lemma \ref{Lyu result on bass number} together with Proposition \ref{Injective satisfy} implies that
	$$
	\mu_j(P, M_{(a,b)}) = \mu_0(P, H^j_P(M_{(a,b)})).
	$$
	By Grothendieck vanishing theorem $H^j_P(M_{(a,b)}) = 0$ for all $j > \dim M_{(a,b)}$, see \cite[6.1.2]{BS}.  So $\mu_j(P, M_{(a,b)}) = 0$ for all $j > \dim M_{(a,b)}$. The result follows.
\end{proof}

\begin{theorem}\label{injdim-and-dim-gen}(with hypotheses as in \ref{Inj dim setup}). Then we have
	\begin{enumerate}[\rm (1)]
		\item
		$\injdim M_{(a,b)} \leq \dim M_{(a,b)}$ for all $(a,b) \in \mathbb{Z}^2$.
		\item
		$\injdim M_{(a,b)} = \injdim M_{(0,0)}$ for all $(a,b)\geq (0,0)$.
		\item $\injdim M_{(a,b)} = \injdim M_{(-m,0)}$ for all $(a,b) \in \mathbb{Z}^2$ with $a \leq -m$ and $b \geq 0$.
		\item $\injdim M_{(a,b)} = \injdim M_{(-m,-n)}$ for all $(a,b)\leq (-m,-n)$.
		\item $\injdim M_{(a,b)} = \injdim M_{(0,-n)}$ for all $(a,b) \in \mathbb{Z}^2$ with $a \geq 0$ and $b \leq -n$.
		\item
		$\dim M_{(a,b)} = \dim M_{(0,0)}$ for all $(a,b)\geq (0,0)$.
		\item
		$\dim M_{(a,b)} = M_{(-m,0)}$ for all $(a,b) \in \mathbb{Z}^2$ with $a \leq -m$ and $b \geq 0$.
		\item
	    $\dim M_{(a,b)} = M_{(-m,-n)}$ for all $(a,b)\leq (-m,-n)$.
	    \item
	    $\dim M_{(a,b)} = M_{(0,-n)}$ for all $(a,b) \in \mathbb{Z}^2$ with $a \geq 0$ and $b \leq -n$.
	\end{enumerate}
\end{theorem}
\begin{proof} 
From the proof of Theorem \ref{Bass th for 1st quad}, we get $\dim_KV_{(a,b)}=\alpha_{(a,b)}=\mu_j(P,M_{(a,b)})$. Here $V$ is a $F$-finite, $F$-module over $K[x_1,\ldots,x_m,y_1,\ldots,y_n]$ and $K$ is infinite.
	
    (1) This follows from Lemma \ref{injdim-dim}.
	
	Let $P$ be a prime ideal in $A$.  Let $(a,b)\geq (0,0)$.
	
	(2) Fix $j \geq 0$. By \ref{Rigidity for A-1},  $\mu_j(P, M_{(0,0)}) > 0$ if and only if $\dim_KV_{(0,0)}\neq 0$  if and only if $\dim_KV_{(a,b)}\neq 0$  for all $(a,b)\geq (0,0)$ if and only if $\mu_j(P, M_{(a,b)}) > 0$ for all $(a,b) \geq (0,0)$.
	The result follows.

    (6) We note that $\M_P$ is $F$-finite, $F$-module over $S=A_P[x_1,\ldots, x_m, y_1, \ldots, y_n]$. By Theorem \ref{Rigidity for A-1} it follows that $(M_{(0,0)})_P \neq 0$ if and only if $(M_{(a,b)})_P \neq 0$. The result follows.
	
	(3), (4), (5), (7), (8), (9) follow with similar arguments as in (2) and (6).
\end{proof}
\begin{picto}
The numbering of the figures corresponds to the parts of the theorem stated. The values of $\dim M_{(a,b)}$ and $\injdim M_{(a,b)}$  in the shaded region can be determined by computing these dimensions for a corresponding red dot in the figure.

    \begin{center}
	\begin{tikzpicture}[scale=0.12]
	\draw[->] (-8.5,0)--(8.5,0) node[right]{$x$};
	\draw[->] (0,-8.5)--(0,8.5) node[above]{$y$};
	\draw[red, dashdotted] (8,-3)--(-7.5,-3)node[left]{$y=-n$};
	\draw[red, dashdotted] (-2,-7.5)--(-2,8) node[left]{$x=-m$};
	\draw[-,black](0,8.5)--(0,0)-- (8.5,0);
	\draw[fill=brown,fill opacity=0.35,draw=none] (0,8.5)--(0,0)--(8.5,0)--(8.5,8.5)--(0,8.5);
	\node at (4.5,3){$\bullet$};
	\node[red] at (0,0){\small{$\bullet$}};
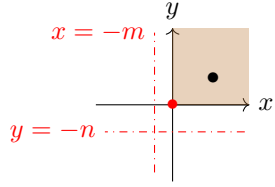
	\node at (0,-11) {\footnotesize{\textit Figure $(2), (6)$}};
	\end{tikzpicture}
	\hspace{2cm}
	\begin{tikzpicture}[scale=0.12]
	\draw[->] (-8.5,0)--(8.5,0) node[right]{$x$};
	\draw[->] (0,-8.5)--(0,8.5) node[above]{$y$};
	\draw[red,dashdotted] (8,-3)--(-7.5,-3)node[left]{$y=-n$};
	\draw[red, dashdotted] (-2,-7.5)--(-2,8) node[left]{$x=-m$};
	\draw[-,black](-2,8.5)--(-2,0)-- (-8.5,0);
	\draw[fill=brown,fill opacity=0.35,draw=none] (-2,8.5)--(-2,0)--(-8.5,0)--(-8.5,8.5)--(-2,8.5);
	\node at (-4.5,3){$\bullet$};
	\node[red] at (-2,0){\small{$\bullet$}};
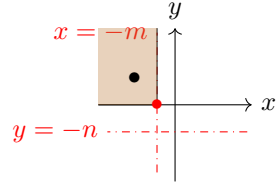
	\node at (0,-11) {\footnotesize{\textit Figure $(3), (7)$}};
	\end{tikzpicture}
	\hspace{0cm}
\end{center}
\vspace{0.5cm}
\begin{center}
\begin{tikzpicture}[scale=0.12]
	\draw[->] (-8.5,0)--(8.5,0) node[right]{$x$};
	\draw[->] (0,-8.5)--(0,8.5) node[above]{$y$};
	\draw[red,dashdotted] (8,-3)--(-7.5,-3)node[left]{$y=-n$};
	\draw[red, dashdotted] (-2,-7.5)--(-2,8) node[left]{$x=-m$};
	\draw[-,black](-2,-8.5)--(-2,-3)-- (-8.5,-3);
	\draw[fill=brown,fill opacity=0.35,draw=none] (-2,-8.5)--(-2,-3)--(-8.5,-3)--(-8.5,-8.5)--(-2,-8.5);
	\node at (-4.5,-6){$\bullet$};
	\node[red] at (-2,-3){\small{$\bullet$}};
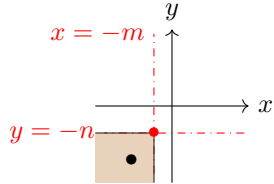
	\node at (0,-11) {\footnotesize{\textit Figure $(4), (8)$}};
	\end{tikzpicture} 
    \hspace{2cm}
	\begin{tikzpicture}[scale=0.12]
	\draw[->] (-8.5,0)--(8.5,0) node[right]{$x$};
	\draw[->] (0,-8.5)--(0,8.5) node[above]{$y$};
	\draw[red,dashdotted] (8,-3)--(-7.5,-3)node[left]{$y=-n$};
	\draw[red, dashdotted] (-2,-7.5)--(-2,8) node[left]{$x=-m$};
	\draw[-,black](0,-8.5)--(0,-3)-- (8.5,-3);
	\draw[fill=brown,fill opacity=0.35,draw=none] (0,-8.5)--(0,-3)--(8.5,-3)--(8.5,-8.5)--(0,-8.5);
	\node at (4.5,-6){$\bullet$};
	\node[red] at (0,-3){\small{$\bullet$}};
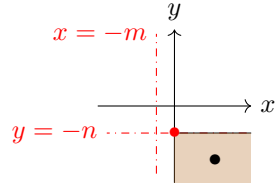
	\node at (0,-11) {\footnotesize{\textit Figure $(5), (9)$}};
	\end{tikzpicture}				
\end{center}
\end{picto}
\section{An Application}
Let $A=K[[z_1,\ldots,z_d]]$ where $K$ is infinite field of characteristic $p>0$. Let $R=A[x_1,\ldots,x_m,y_1,\ldots,y_n]$ be standard bigraded. Consider $S=R/I$ where $I$ is a bi-homogeneous ideal of $R$. Let $R_{++}$ be the ideal of $R$ generated by the elements $x_iy_j$ for all $i=1,\ldots,m$ and $j=1,\ldots,n$. By $\Bproj(R)$, we denote the bigraded prime ideals of $P$ of $R$ such that $P \not\supseteq R_{++}$. Let $\mathfrak{n}$ be unique bigraded homogeneous ideal of $S$. 

The next result offers a convenient criterion for determining when $M_{(m,n)}=0$ for all $(m,n)\geq (0,0)$, where $M$ is a bigraded $F$-finite, $F$-module over $R$. Although our primary focus is the case where $R_0=k[[z_1,\ldots,z_d]]$, it is advantageous to prove a more general version of the result.
\begin{lemma}\label{app-lemma}
 Let $A$ be a regular domain of dimension $d$ containing an infinite field of characteristic $p>0$. Let $R=A[x_1,\ldots,x_m,y_1,\ldots,y_n]$ be standard bigraded over $A$, i.e., $\bideg(A)=(0,0)$, $\bideg(x_i)=(1,0)$ and  $\bideg(y_j)=(0,1)$ for all $i$ and $j$ and  $M=\bigoplus_{i,j} M_{(i,j)}$, a bigraded $F_R$-finite, $F_R$-module. The following statements are equivalent:
 \begin{enumerate} [\rm (1)]
     \item $\Gamma_{R_{++}}(M)=M$.
     \item $M_{(m,n)}=0$ for $(m,n)\geq(0,0)$.
     \item $M_{(m,n)}=0$ for $m\gg0$ and $n\gg0$.
     \item If $P$ is an associated prime of $M$ then $P\supseteq R_{++}$.
 \end{enumerate}
\end{lemma}
\begin{proof}
    $(1)\implies (2)$ We prove it by induction on $\dim A$. If $\dim A=0$, then A is field say $K$. First assume that $K$ is uncountable. If $M_{(0,0)}=0$, then we are done by \ref{Rigidity for K-1}. Hence assume that $M_{(0,0)}\neq 0$. For a fixed $b\geq 0$, consider $N_b=\oplus_{i\in \mathbb{Z}} M_{(i,b)}$. Clearly $N_b$ is Eulerian over $T=K[x_1,\ldots,x_m]$. We also note that $\Ass_T N_b$ is finite. Consider the following exact sequence $$0\rightarrow \Gamma_{(x_1,\ldots,x_m)}(N_b)\rightarrow N_b\rightarrow \overline{N_b}\rightarrow 0.$$
    Since $\Ass_T N_b$ is finite for all $b\geq 0$ so $\bigcup\limits_{b\geq 0}\Ass_T N_b$ is countable set. By \ref{Ass of M/I},  $\bigcup\limits_{b\geq 0}\Ass_T \overline{N_b}$ is also countable. Let $V=Kx_1\oplus\ldots\oplus Kx_m$. If $P\in \bigcup\limits_{b\geq 0}\Ass_T \overline{N_b}$, then $P\not\supseteq(x_1,\ldots,x_m)$. Therefore $P\cap V$ is proper subspace of $V$. As $K$ is uncountable there exist $$\xi\in V
    \ \backslash\bigcup\limits_{P\in \bigcup\limits_{b\geq 0}\Ass_T \overline{N_b}}(P\cap V).$$ Clearly $\xi$ is of bidegree $(1,0)$ such that it is $\overline{N_b}$ regular for all $b\geq 0$. Since $(N_b)_u=(\overline{N_b})_u$ for $u\geq 0$ so $\xi$ is regular on $\oplus_{i\geq 0}M_{(i,b)}$.
    Now for a fixed $a\geq 0$, consider $U_a=\oplus_{j\in \mathbb{Z}} M_{(a,j)}$. By a similar argument we get $\eta$ of bidegree $(0,1)$ such that it is regular on $\oplus_{j\geq 0}M_{(a,j)}$. But this implies that $\xi^u\eta^v: M_{(0,0)}\rightarrow M_{(u,v)}$ is injective for all $(u,v)\geq (0,0)$ which contradicts the fact that $M$ is $R_{++}$-torsion. 

    Now suppose that $K$ is countable and $M_{(0,0)}\neq 0$. Let $K'$ be an uncountable field containing $K$. Let  $S=K'[x_1,\ldots,x_m,y_1,\ldots,y_n]=R\otimes_KK'$. Then $N=M\otimes_R S=M\otimes_K K'$ is $F$-finite, $F$-module over $S$. Clearly $N$ is $S_{++}$-torsion and $N_{(0,0)}\neq 0$ which is a contradiction as before. Hence the proof.
    
    Now we assume the result for $\dim A<r$ and prove for $\dim A=r$. If possible let $M_{(0,0)}\neq 0$ and let $P$ be its minimal prime. Let $B=A_P$ and $R_B=B[x_1,\ldots,x_m,y_1,\ldots,y_n]$. Then $N=M\otimes_R R_B=M\otimes_AB$ is bigraded $F_{R_B}$-finite, $F_{R_B}$-module. Now $N$ is $(R_B)_{++}$- torsion and $N_{(0,0)}\neq 0$. If $\dim B<r$ then we have contradiction by induction hypothesis. Assume $\dim B=r$. Now $N_{(0,0)}$ has structure of $C=\widehat{A_P}$-module. Then $L=N\otimes_{R_{B}} R_C=N\otimes_B C$ is bigraded $F_{R_C}$-finite, $F_{R_C}$-module where $R_C=C[x_,\ldots,x_m,y_1,\ldots,y_n]$. Now $L$ is $(R_{C})_{++}$-torsion and $L_{(0,0)}\neq 0$. Let $C=K'[[z_1,\ldots,z_r]]$. Consider the exact sequence $$0\rightarrow H_1(z_r,L)\rightarrow L\xrightarrow{z_r}L\rightarrow H_0(z_r,L)\rightarrow0.$$
    Let $\overline{C}=C/z_rC$ and $R_{\overline{C}}=\overline{C}[x_1,\ldots,x_m,y_1,\ldots,y_n]$. Therefore, $T_i=H_i(z_r,L)$ are bigraded $F_{R_{\overline{C}}}$-finite, $F_{R_{\overline{C}}}$-module for $i=0,1$ and also $(R_{\overline{C}})_{++}$-torsion. Therefore, by the induction hypothesis $(T_i)_{(m,n)}=0$ for $(m,n)\geq (0,0)$. Hence multiplication map by $z_r$ on $L_{(0,0)}$ is isomorphism and therefore $L_{(0,0)}$ is $C_{z_r}$-module. Let $S=C_{z_r}[x_1,\ldots,x_m,y_1,\ldots,y_n]$. Now $U=L\otimes_{R_C}T=L\otimes_CC_{z_r}$ is $F_T$-finite, $F_T$-module and $T_{++}$-torsion. But $\dim C_{z_r}=r-1$ and $U_{(0,0)}\neq 0$ contradict the induction hypothesis. Therefore $M_{(0,0)}=0$ so the result follows by \ref{Rigidity for K-1}.

    $(2)\iff (3)$ follows from \ref{Rigidity for K-1}.

    $(3)\implies (4)$. Let $M_{(m,n)}=0$ for $m\geq m_0 $ and  $n\geq n_0$. We claim that $\Gamma_{R_{++}}(M)=M$. Let $a\in M$ such that $\bideg a =(r,s)$. Choose $l\gg 0$ such that $(r+l,s+l)>(m_0,n_0)$. Then, $(R_{++})^lm=0$ and this implies that  $\Gamma_{R_{++}}(M)=M$. Let $P\in \Ass M$. We have $P=(0:x)$ for some $x(\neq 0)\in M=\Gamma_{R_{++}}(M)$. Therefore, $(R_{++})^sx=0\implies (R_{++})^s\subseteq P\implies R_{++}\subseteq P$. 

$(4)\implies(1)$ We have an exact sequence $$0\rightarrow \Gamma_{R_{++}}(M)\rightarrow M\rightarrow N\rightarrow 0.$$
By assumption $\Ass N=\phi$ since \ref{Ass of M/I} implies $\Ass N$ are those of $M$ which does not contain $R_{++}$. So $N=0$. Therefore $M=\Gamma_{R_{++}}(M)$.
\end{proof}
The following lemma is useful;
\begin{lemma}\label{E_P=0}
    Let $E$ be a finitely generated bigraded $S$-module. Then the following statements are equivalent:
    \begin{enumerate}[\rm (1)]
        \item $H^0_{S_{++}}(E)=E$.
        \item $E_P=0$ for all $P\in \Bproj(S)$.
    \end{enumerate}
\end{lemma}
\begin{proof}
    $(1)\implies (2)$ is trivial. For $(2)\implies (1)$, if possible assume that $H^0_{S_{++}}(E)\neq E$. Consider the following exact sequence 
    $$0\rightarrow \Gamma_{S_{++}}(E)\rightarrow E\rightarrow E/{\Gamma_{S_{++}}(E)}\rightarrow 0.$$
    Let $U=E/{\Gamma_{S_{++}}(E)}$. By \ref{Ass of M/I}, $\Ass U=\{P\in \Ass E\  |\  P\not\supseteq S_{++}\}$. We note that if $P\in \Ass E$, then $P$ is bigraded. Let $P\in \Ass U$ , so this implies $P\in \Bproj(S)$. Therefore $U_P\neq 0$ and this implies $E_P\neq 0$, a contradiction.
\end{proof}
\begin{definition}
    We define $$\gdepth S=\sup\{k\in \mathbb{Z}\ |\ S_{++}\subseteq \sqrt{\ann H^i_{\mathfrak{n}}(S)}\ \text{for all}\  i<k\}$$
\end{definition}
We generalize a result from \cite[Proposition 2.1]{Marley} to the bigraded setup. We mostly follow their proof. For the convenience of the reader, we include it here.
\begin{theorem}\label{G depth formula}
    $\gdepth(S)=\min_{P\in \Bproj(S)}\{\depth S_P+\dim S/P\}$.
\end{theorem}
\begin{proof}
We are given that $S=R/I$. Let $d=\dim R$. By local duality theorem of graded modules,
$H^i_\mathfrak{n}(S)=\Ext^{d-i}_R(S,R)^\vee$ where $\vee$ is Matlis dual functor. Since $\ann E=\ann E^\vee$, so some power of $S_{++}$ annihilates $H^i_\mathfrak{n}(S)$ iff some power of $S_{++}$ annihilates $\Ext^{d-i}_R(S,R)$. If $P\in \Bproj(S)$, then there is $P'\in\Bproj(R) $ containing $I$ such that $P'/I=P$. Let $r=\dim R_{P'}$. Then $d=\dim R=\dim R_{P'}+\dim R/P'=r+\dim S/P$. Now \begin{align*}
\Ext^{d-i}_R(S,R)_P=0 &\iff \Ext^{d-i}_{R_{P'}}(S_P,R_{P'})=0\\ &\iff\Ext_{R_{P'}}^{r-(i-\dim S/P)}(S_P,R_{P'})^\vee=0\\ &\iff H^{i-\dim S/P}_{PS_P}(S_P)=0
\end{align*}
Let $s=\min_{P\in \Bproj(S)}\{\depth S_P+\dim S/P\}$. Assume $\gdepth(S)\geq k$. Then some power of $S_{++}$ annihilates $H^i_{\mathfrak{n}}(S)$ for $0\leq i\leq k-1$. Therefore $\Ext^{d-i}_R(S,R)_P=0\iff H^{i-\dim S/P}_{PS_P}(S_P)=0$ for $0\leq i\leq k-1$. Therefore, $\depth S_P\geq k-\depth S/P$ for all $P\in \Bproj(S)$ and this proves that $\gdepth(S)\leq s$.

Conversely assume $\depth S_P+\dim S/P\geq k$ for all $P\in \Bproj(S)$. Then $\Ext^{d-i}_R(S,R)_P=0 $ for all $0\leq i\leq k-1$ and for all $P\in \Bproj(S)$. So by Lemma \ref{E_P=0}, $S_{++}\subseteq \ann \Ext^{d-i}_R(S,R)$. Therefore $S_{++}\subseteq \ann H^i_\mathfrak{n}(S)$ for $0\leq i\leq k-1$. Hence $s\leq \gdepth(S)$.
\end{proof}
\s Let $\phi:S\rightarrow S$ be Frobenius map. Then it induces a Frobenius map $f$ on $H^i_\mathfrak{n}(S)$ for $i\geq 0$ and $f(H^i_\mathfrak{n}(S))_{(m,n)}\subseteq H^i_\mathfrak{n}(S)_{(mp,np)}$ for all $(m,n)\in \mathbb{Z}^2$. By $H^i_\mathfrak{n}(S)^{f^t}$ we denote the bigraded $S$-submodule of $H^i_\mathfrak{n}(S)$ generated by $f^t(H^i_\mathfrak{n}(S))$. So we have a descending chain $H^i_\mathfrak{n}(S)\supseteq H^i_\mathfrak{n}(S)^f\supseteq H^i_\mathfrak{n}(S)^{f^2}\supseteq\ldots$
Since $H^i_\mathfrak{n}(S)$ is bigraded $*$-Artinian this chain stabilizes and we denote the stable value by $H^i_\mathfrak{n}(S)^*$.
\begin{definition} Define,
\begin{enumerate}
    \item $\Bfg(S)=\max\{k\ |\ \text{there exist} \ (a,b)\leq (0,0) \ \text{such that }\ H^i_\mathfrak{n}(S)_{(m,n)}=0\ \text{for all}\ (m,n)\leq (a,b)\ \text{and for all}\ i<k\}$.
   \item$\Bfgt(S)=\max\{k\ |\ \text{there exist} \ (a,b)\leq (0,0) \ \text{such that }\ H^i_\mathfrak{n}(S)^*_{(m,n)}=0\ \text{for all}\  (m,n)\leq (a,b)\ \text{and for all}\ i<k\}$.
   \end{enumerate}
\end{definition}
Clearly $\Bfg(S)\leq \Bfgt(S)$. 
\begin{theorem}\label{Bfg=gdepth}
    $\Bfg(S)=\gdepth(S)$.
\end{theorem}
\begin{proof}
    We have $\gdepth S=\sup\{k\in \mathbb{Z}\ |\ S_{++}\subseteq \sqrt{\ann H^i_{\mathfrak{n}}(S)}\ \text{for all}\  i<k\}$ and $H^i_\mathfrak{n}(S)=\Ext^{d-i}_R(S,R)^\vee$. Let $i< \gdepth(S)$. Therefore some power of $S_{++}$ annihilates $H^i_\mathfrak{n}(S)$. But since $\ann \Ext^{d-i}_R(S,R)^\vee=\ann \Ext^{d-i}_R(S,R)$ , some power of $S_{++}$ annihilates $\Ext^{d-i}_R(S,R)$. This implies that there exists $(a,b)\geq (0,0)$ such that $\Ext^{d-i}_R(S,R)_{(m,n)}=0$ for all $(m,n)\geq (a,b)$. Therefore $H^i_\mathfrak{n}(S)_{(m,n)}=0$ for all $(m,n)\leq (-a,-b)$. This implies $\Bfg(S)> i$. Therefore $\Bfg(S)\geq \gdepth(S)$.

    Now assume that $k=\Bfg(S)$. If $i<k$, then there is $(a,b)\leq (0,0)$ such that $H^i_\mathfrak{n}(S)_{(m,n)}=0$ for all $(m,n)\leq (a,b)$. This implies $H^i_\mathfrak{n}(S)_{(m,n)}^\vee=0$ for all $(m,n)\geq (-a,-b)$ and $(-a,-b)\geq (0,0)$. So some power of $S_{++}$ annihilates $\Ext^{d-i}_R(S,R)\implies S_{++}\subseteq  \sqrt{\ann H^i_{\mathfrak{n}}(S)}$ for all $0\leq i\leq k-1$.  Therefore $\gdepth(S)\geq k$.
\end{proof}
Having proved several preparatory results, we now state the following corollary, which is crucial for proving the main result of this section (See Theorem \ref{Bproj cohen}).
\begin{corollary}\label{bfg=dims}
    If $S$ is equidimensional, $\Bproj(S)\neq \phi$ and Cohen-Macaulay, the $\Bfg(S)=\dim S$.
\end{corollary}
\begin{proof}
    The proof follows from \ref{G depth formula} and \ref{Bfg=gdepth}.
\end{proof}
\begin{theorem}\label{app-2}
    Let $i<\Bfgt(S)$. Then, $H^{\dim R-i}_I(R)_{(m,n)}=0$ for all $(m,n)\geq (0,0)$.
\end{theorem}
\begin{proof}
Let $i<\Bfgt(S)$. So there exists some $(a,b)\leq (0,0)$ such that $ H^i_\mathfrak{n}(S)^*_{(m,n)}=0$ for all $(m,n)\leq (a,b)$. Let $E_i$ be a bigraded root of $H^{\dim R-i}_I(R)$ such that $E_i^\vee= H^i_\mathfrak{n}(S)^*$. So $(E_i)_{(m,n)}=( H^i_\mathfrak{n}(S)^*)^\vee_{(m,n)}=0$ for all $(m,n)\geq (-a,-b)$. Therefore by \ref{app-lemma}, if $P$ is a associated prime of $E_i$, it contains $R_{++}$. But since $\Ass E_i=\Ass H^{\dim R-i}_I(R) $, the result follows from \ref{app-lemma}.
\end{proof}
Finally, we prove the main result of this section as follows:\begin{theorem}\label{Bproj cohen}
Let $(A,\mathfrak{m})$ be regular local containing a field of characteristic $p$. Let $R=A[x_1,\ldots,x_m,y_1,\ldots,y_n]$ be standard bigraded. Consider $S=R/I$ where $I$ is a bi-homogeneous ideal of $R$. Assume that $S$ is equidimensional, $\Bproj(S)\neq \phi$ and Cohen-Macaulay. Then, $H^j_I(R)_{(m,n)}=0$ for all $(m,n)\geq (0,0)$ and all $j>\height I$.
\end{theorem}
\begin{proof}
We may assume that $A$ is complete and the field it contains is infinite. Clearly $\Bfg(S)=\dim S$ by \ref{bfg=dims}. Let $j>\height I$, then $j=\height I+t$ for some $t\in \mathbb{N}$. Since $\height I=\dim R-\dim S\implies \height I+t=\dim R-(\dim S-t)$. Now the result follows from \ref{app-2}.
\end{proof}
\section*{Acknowledgement}
The first author is grateful to the Government of India for the support through the Prime Minister's Research Fellowship (PMRF ID 1303161) during the course of this work. He also thanks Dr. Sudeshna Roy for helpful conversations and for assisting in drawing the diagrams for this paper.

\providecommand{\bysame}{\leavevmode\hbox to3em{\hrulefill}\thinspace}
\providecommand{\MR}{\relax\ifhmode\unskip\space\fi MR }
\providecommand{\MRhref}[2]{
  \href{http://www.ams.org/mathscinet-getitem?mr=#1}{#2}
}

\end{document}